\documentclass[final]{article}
\usepackage{amsfonts,amsmath}
        \newtheorem{theorem}{Theorem}

        \newtheorem{lemma}[theorem]{Lemma}
        \newtheorem{corollary}[theorem]{Corollary}
        \newtheorem{remark}[theorem]{Remark}

\usepackage{epsfig,color}

\usepackage[sort&compress,numbers,square]{natbib}

\title{Asymptotic expansions for high-contrast elliptic equations}

\author{ 
Victor M. Calo\thanks{ Applied Mathematics \& Computational Science and Earth Sciences \& Engineering, King Abdullah University of Science and Technology, Thuwal 23955-6900, Kingdom of Saudi Arabia}
\and Yalchin Efendiev  
\thanks{Department of Mathematics, Texas A \& M University, College Station, TX 77843
 (\tt{efendiev@math.tamu.edu})}
\and
Juan Galvis\thanks{Department of Mathematics, Texas A \& M University, College Station, TX 77843 ({\tt jugal@math.tamu.edu})}
}

\date{}
\begin{document}
\maketitle
%\begin{abstract}
%fasdfadf\\
%{\bf Keywords.} 
% asdfasdf\\
%{\bf AMS subject classifications.}
%\end{abstract}

\pagestyle{myheadings}
\thispagestyle{plain}
\markboth{hc.tex}{}

%%%%%%%%%%%%%%%%%%%%%%%%%%%%%%%%%%%%%%%%%%%%%%%%%%%%%%%%%%%%%%%%%%%%
%%%%%%%INTRODUCTION AND PROBLEM SETTING%%%%%%%%%%%%%%%%%%%%%%%%%%%%%
%%%%%%%%%%%%%%%%%%%%%%%%%%%%%%%%%%%%%%%%%%%%%%%%%%%%%%%%%%%%%%%%%%%%
%%%%%%%%%%%%%%%%%%%%%%%%%%%%%%%%%%%%%%%%%%%%%%%%%%%%%%%%%%%%%%%%%%%%
%\tableofcontents 
\section{Abstract}

In this paper, we present a high-order expansion for elliptic equations in 
high-contrast media. The background conductivity 
is taken to be one and we assume the medium contains
high (or low) conductivity inclusions.
We derive an asymptotic expansion with 
respect to the contrast and provide a procedure
to compute the terms in the expansion. 
The computation of the expansion does not depend
on the contrast which is important for
simulations. The latter allows avoiding 
increased mesh resolution around high
conductivity features. 
This work is partly motivated by our
earlier work in~\cite{ge09_1} where we design
efficient numerical procedures for solving
high-contrast problems. These multiscale approaches 
require local solutions and our proposed high-order
expansion can be used to approximate these local solutions inexpensively.
In the case of a large-number of inclusions,
the proposed analysis can help to design localization
techniques for computing the terms in the expansion.
In the paper, we present a rigorous analysis
of the proposed high-order expansion and estimate
the remainder of it. We consider both high and
low conductivity inclusions. 

\section{Introduction}
The mathematical analysis and numerical analysis of partial 
differential equations 
in high-contrast and multiscale media are important for many 
practical applications. For instance, in porous media applications, 
the permeability of  subsurface regions is described 
as a quantity with 
high-contrast and multiscale features.
A main goal is to understand the 
effects and complexity related to this multiscale variation 
and high-contrast in the coefficients.
This is specially important 
for the computation of numerical solutions and 
quantities of interest.
Many tools and methods have been developed and used to study 
high contrast problems. We mention  few recent 
works where  numerical methods are designed that target 
problems with difficult variation in the coefficients.
The numerical analysis of these methods require 
accurate descriptions of the variation of the coefficients.
In~\cite{egw10,ceg10,cgh09,ce10,oz11,bl10,bo09_1,bo09,vsg,AMG2} 
 multiscale methods for problems with high-contrast 
coefficients are described.
For domain decomposition methods with discontinuous 
coefficients  we mention~\cite{marcus1,Drjya,Nepom-91, tw,Tarekbook} 
and references 
therein. Some domain decomposition techniques for 
multiscale partial differential equations with complicated 
variations in the coefficients are developed
in~\cite{nataf11,eglw11,Graham1,rob_clement,ge09_1,ge09_2,ge09_3}.
Additionally, we mention~\cite{AMG1,egvdd20,XuZikatanov} among others, 
which focus on multilevel methods targeting problems with 
high-contrast and discontinuous coefficients. 
Numerical methods based on different asymptotic 
analysis  are detailed 
in~\cite{BerlyandNovikov, EILRW08,klapper,DangQuang,bp98}.

In the case of elliptic problems with oscillating coefficients and 
bounded contrast, multiscale 
methods and homogenizations techniques have been 
 successfully applied to study solutions of elliptic 
differential equations (c.f.,~\cite{eh09,blp78,pavliotis-stuart}
 and references therein). 
Many homogenization and multiscale methods start with 
the derivation of an asymptotic expansion for  the 
solution of the partial differential equation.
The expansion is written in terms of the problem parameter(s): e.g., 
the period in the case of oscillating periodic coefficients.
The asymptotic expansion is then used to study 
the problem at hand. 

In this paper and in the same spirit, 
we derive  asymptotic expansions for the solutions 
of elliptic problems with high-contrast. In this case 
the parameter to be consider is the contrast in the coefficient. 
In particular, we consider the problem
\begin{equation}\label{eq:problem-strong}
-\mbox{div}(\kappa(x)\nabla u )=f, \mbox{   in } D
\end{equation}
with Dirichlet data given by $u=g$ on $\partial D$.  
We assume that $\kappa(x)>0$. 
The contrast 
in the coefficient, 
$\|\kappa\|_{L^\infty(D)}\|\kappa^{-1}\|_{L^\infty(D)}$,
is the important parameter considered here.
For the analysis, we
 consider a binary media $\kappa(x)$ with background one and with 
multiple (connected) inclusions. We consider 
high-conductivity inclusions (with conductivity $\eta$) and 
low-conductivity inclusions  (with conductivity $1/\eta$). 
We derive expansions of the form 
\begin{equation}\label{eq:expansion-general}
u_{\eta}=\eta u_{-1}+u_{0}+ \frac{1}{\eta} u_1+
\frac{1}{\eta^2}u_2+\dots.
\end{equation}
In the case with only high-conductivity inclusions we 
have $u_{-1}=0$ and~\eqref{eq:expansion-general} reduces to 
\begin{equation}\label{eq:expansion-general-reduced}
u_{\eta}=u_{0}+ \frac{1}{\eta} u_1+
\frac{1}{\eta^2}u_2+\dots.
\end{equation}
In the presence of low-conductivity 
inclusions we may have $u_{-1}\not=0$ depending on 
the support of the forcing term.
We mention here that, in order to 
derive the expansions, we use the weak formulation 
associated to~\eqref{eq:problem-strong}. 
In this case, using the integral formulation has 
the advantage that the boundary, interface and 
transmission conditions are self-revealing.

The asymptotic problems for the case 
of only high-conductivity or only low-conductivity are studied in 
detail. 
For the study of the high-conductivity inclusions asymptotic 
problem we use harmonic characteristic functions.
These functions are defined 
as being constant inside the inclusions and harmonic in 
the background domain. The asymptotic 
solution can be obtained by solving a Dirichlet 
problem in the background domain and a 
finite dimensional problem in the space
spanned by the harmonic characteristic functions. 
The solution of the finite dimensional problem 
gives a closed formula for the constant values of the limit 
solution inside the high-conductivity inclusions.
The resulting system can be large, in general, and one can 
consider some localization techniques 
(c.f.,~\cite{BerlyandNovikov,yulya-berlyand}).

The asymptotic problem  and  approximations of  
the asymptotic problem, in the case of high-conductivity 
inclusions, have been studied in the literature.
We mention~\cite{BerlyandNovikov,yulya-berlyan2,bp98} where a
discrete network approximation is considered for the 
problem of computing the effective conductivity of high-contrast, 
randomly, and densely packed composites with 
high-conductivity inclusions.
The network approximation depends on the geometry of the inclusions
and the behavior of the solution between nearby inclusions.
The authors can localize the interaction of high-conductivity inclusions
using graph-theoretical concepts. Furthermore, they propose a finite element
approach for solving the resulting system and identifying the first order approximation of the solution.  
In~\cite{yulya-berlyand} the authors study the 
homogenization of the asymptotic problem 
in terms of geometric parameters such as the shapes of inclusions and the distance between the inclusions.  
They develop an 
asymptotic analysis for periodic structures 
with absolutely conductive square inclusions. The small scales 
considered here are the period of the structure 
and the distance between inclusions. 
We refer to the works~\cite{BerlyandNovikov,yulya-berlyand, yulya-berlyan2} 
and references therein.

We also write 
a low-conductivity asymptotic problem valid
only in the case where the forcing term vanishes 
inside the low-conductivity inclusions. 
This problem is important in flow applications where
low conductivity regions represent shale regions and can 
substantially alter the overall flow behavior.
To our best knowledge, this problem 
 is not extensively studied in the literature.
The asymptotic 
solution can be obtained by solving: 1) a 
mixed boundary condition problem 
in the background domain, and then, 2) a Dirichlet 
problem inside the inclusion with zero forcing term 
and the Dirichlet data from 1).

We show how to 
obtain  all the coefficients in the expansions.  
The procedure 
to compute the coefficients, coincide with a 
Dirichlet-to-Neumann procedure as in Domain Decompositions Methods; 
see~\cite{Tarekbook,tw}. 
For the case of high-conductivity inclusions, the Neumann 
problems are solved in the interior inclusions and therefore, 
a compatibility condition needs to be verified. 
Obtaining the right balance of fluxes for the 
compatibility condition involves the solution of 
a finite dimensional problem in the space 
spanned by the harmonic characteristic functions 
mentioned above. For the case of low-conductivity inclusions the Neumann 
problems are solved in the background domain and 
no compatibility condition in required. The expansions derived in this paper are proven to converge 
in $H^1(D)$ for high-contrast bigger that a certain constant.
This constant  depends on the domains representing the inclusions 
and the background domain.  Asymptotic expansions in the 
presence of  boundary intersecting inclusions  can  be derived and analyzed using similar arguments. The presence 
of low- and high-conductivity inclusions can be also analyzed. 
More general geometrical configurations and partial differential 
equations can be studied as well.

Having  a practical procedure to compute 
the next leading order terms in~\eqref{eq:expansion-general} 
is useful for applications. For instance, the quantity 
$u_{1}$ in~\eqref{eq:expansion-general-reduced} may have considerable contribution to some quantity of interest in some regions, e.g., in the velocity 
$\kappa|\nabla u_{\eta}|$ expansion, the solution is multiplied by $\eta$.
A high-order expansion is also useful when constructing multiscale
and multilevel methods. 
The expansion~\eqref{eq:expansion-general} can be used 
to construct multiscale finite element basis functions; see 
\cite{eh09, cgh09}. 
Such an expansion will allow the construction of basis 
functions independent of the contrast and depending only 
on the limiting problem. In the case of expansion~\eqref{eq:expansion-general},
the asymptotic 
problem depends only on the geometry configuration describing 
the inclusions (see Section~\ref{sec:asymp-problem-multiple}). 
These basis functions will capture the effect on the solution 
of the geometric arrangement describing the conductivity. 
The next order terms in~\eqref{eq:expansion-general} 
can be used to construct correction terms to account 
for the effect of the contrast 
in the coefficient. Expansion~\eqref{eq:expansion-general} 
can be used to improve existing advanced multiscale finite element
techniques for a better sub-grid capturing; see~\cite{ceg10,eh09}. 
Fast numerical upscaling techniques can also be constructed 
with the first order terms of~\eqref{eq:expansion-general} 
or~\eqref{eq:expansion-general-reduced}. 
See~\cite{EILRW08} where the authors develop 
fast numerical upscaling methods based on some asymptotic analysis.
We also mention~\cite{klapper,DangQuang} where 
numerical approximations are designed using asymptotic 
analysis.

The rest of the paper is organized as follows. 
In Section~\ref{sec:problem:setting} we 
recall the weak formulation of~\eqref{eq:problem-strong}. 
In Section~\ref{sec:high} we 
derive the expansion  for the case of high-conductivity 
inclusions. We study the asymptotic 
problem and the convergence of the expansions.
Section~\ref{sec:low} is dedicated to the 
case of low-conductivity inclusions. 
In Section~\ref{sec:low-and-high} we consider 
the case with low- and high-conductivity inclusions 
and in Section~\ref{sec:conclusions} we make some 
conclusion and final comments.

\section{Problem Setting}\label{sec:problem:setting}

Let $D\subset\mathbb{R}^d$    
polygonal domain or a domain with smooth boundary.
We consider the following weak formulation of~\eqref{eq:problem-strong}.
Find $u\in H^1(D)$ such that 
\begin{equation}\label{eq:problem}
\left\{\begin{array}{ll}
a(u,v)=f(v) & \forall v\in H_0^1(D),\\
\hspace{.3in} u=g &\mbox{ on } \partial D.
\end{array}\right.
\end{equation}
Here the bilinear form $a$ and the linear functional $f$ are 
defined by
\begin{align}\label{eq:def:a}
a(u,v)&=\int_D 
\kappa(x)\nabla u(x)\cdot  \nabla v(x)  
 &&\forall  u,v\in H_0^1(D)\\
f(v)&=\int_Df(x)v(x) &&\forall v\in H_0^1(D). 
\end{align}
We assume that $D$ is the disjoint union 
of a background domain and inclusions,
$D=D_0\cup (\cup_{m=1}^M D_m).$ We assume 
that $D_0,D_1,\dots,D_M,$ are  
polygonal domains (or domains 
with smooth boundaries). We also assume that 
each $D_m$ is a connected domain, $m=0,1,\dots,M$. 
Let $D_0$ represent the background domain and 
the subdomains $\{D_m\}_{m=1}^M$ represent 
the inclusions. For simplicity of the presentation 
we consider only interior inclusions. See 
Figure~\ref{fig:figure1} for two dimensional 
illustrations. 

\begin{figure}[htb]
\centering
{\psfig{figure=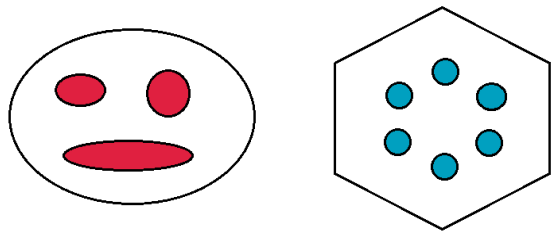,height=3.3cm,width=8cm,angle=0}}
\caption{Examples of geometry configurations with 
interior inclusions.}
\label{fig:figure1} 
\end{figure}

Given $w\in H^1(D)$ we will use the notation $w^{(m)}$, 
for the restriction of $w$ to the domain $D_m$, that is, 
\[
w^{(m)}=w|_{D_m}, \quad m=0,1,\dots,M.
\]

\section{Expansions for high-conductivity inclusions }
\label{sec:high}

In this section we derive and analyze expansions for 
the case of high conductivity inclusions. 
For the sake of readability and presentation, we consider 
first the case of only one high-conductivity 
inclusion in Section~\ref{sec:one}
and study the convergence of this
expansion in Section~\ref{sec:convergence}. 
We present the multiple high-conductivity inclusions 
case in Section~\ref{sec:high-multiple}, where we describe
the expansion and analyze its convergence, following the structure
presented in Sections~\ref{sec:one} and~\ref{sec:convergence}.

\subsection{Derivation for one high-conductivity inclusion}
\label{sec:one}

Let $\kappa$ be defined by 
\begin{equation}\label{eq:coeff1inc}
\kappa(x)=\left\{\begin{array}{cc} 
\eta,& x\in   D_1,  \\
1,& x\in  D_0=D\setminus \overline{D}_1,
\end{array}\right.
\end{equation}
and denote by $u_{\eta}$ the solution
of the weak formulation ~\eqref{eq:problem}.
We assume that 
$D_1$ is compactly included in $D$ ($\overline{D}_1\subset D$).
Since $u_\eta$ is solution of~\eqref{eq:problem} with 
the coefficient~\eqref{eq:coeff1inc}, we have
\begin{equation}\label{eq:problem-with-1-inclusion}
\int_{D_0}\nabla u_{\eta}\cdot  \nabla v +\eta\int_{D_1}\nabla u_\eta
\cdot  \nabla v=
\int_{D}fv \quad \forall v\in H_0^1(D).
\end{equation}

We seek to determine $\{u_i\}_{i=0}^\infty\subset 
H^1(D)$ such that, 
\begin{equation}\label{eq:expansion}
u_\eta=u_0+\frac{1}{\eta}u_1+\frac{1}{\eta^2}u_2+\dots=
\sum_{i=0}^\infty \eta^{-i} u_{i}, 
\end{equation}
and such that they satisfy the following Dirichlet boundary 
conditions, 
\begin{equation}\label{eq:expansionBC}
u_0=g \mbox{ on } \partial D \quad \mbox{ and }
\quad u_{i}=0 \mbox{ on } \partial D \mbox{ for } i\geq 1.
\end{equation}

We substitute~\eqref{eq:expansion} into~\eqref{eq:problem-with-1-inclusion} 
to obtain that for all $v\in H^1_0(D)$ we have, 
\[
\sum_{i=0}^\infty \eta^{-i} \int_{D_0}\nabla u_{i}\cdot  \nabla v+
\sum_{i=0}^\infty \eta^{-i+1}\int_{D_1}\nabla u_{i}\cdot  \nabla v=\int_{D}fv
\]
or 
\begin{equation}
\eta\int_{D_1}\nabla u_{0}\cdot \nabla v+
\sum_{i=0}^\infty 
\eta^{-i}\left(
 \int_{D_0}\nabla u_{i}\cdot \nabla v+\int_{D_1}\nabla u_{i+1}\cdot \nabla v\right)
=\int_{D}fv. \label{eq:WeakSequence}
\end{equation}
Now we collect terms with equal powers of $\eta$ and analyze the resulting 
subdommain equations. 

\subsubsection{Term corresponding to  $\eta=\eta^1$}

In~\eqref{eq:WeakSequence} there is one term corresponding to $\eta$ to the power 1, thus we obtain the following equation 
\begin{equation}
\int_{D_1}\nabla u_{0}\cdot \nabla v=0 \mbox{ for all } v\in H^1_0(D).
\end{equation}
In the general case, 
the meaning of this equation depends on  the relative 
position of the inclusion $D_1$ with respect to the boundary. 
It may need to take the boundary data into account. 
Since we are assuming  that $\overline{D}_1\subset D$, we conclude that 
$\nabla u_0^{(1)}=0$ in $D_1$ and then $u_0^{(1)}$ (the restriction
of $u_0$ on $D_1$) is a constant.

\subsubsection{Terms corresponding to $\eta^0=1$}
Equation~\eqref{eq:WeakSequence} contains three terms corresponding to $\eta$ to the power 0, which are:
\begin{equation}\label{eq:eta0}
 \int_{D_0}\nabla u_{0}\cdot \nabla v+\int_{D_1}\nabla u_{1}\cdot \nabla v=\int_{D}fv 
\mbox{ for all } v\in H^1_0(D).
\end{equation}
Let 
\[
V_{const}=\{ v\in H^1_0(D), \mbox{ such that } 
v^{(1)}=v|_{D_1} \mbox{ is  constant}\}.
\]
If we consider $z\in V_{const}$ in equation~\eqref{eq:eta0} we conclude 
that $u_0$ satisfies the following problem, 
\begin{align}\label{eq:limitproblem}
\displaystyle  \left(\int_{D}\nabla u_{0}\cdot \nabla z= \right) \quad 
\int_{D_0}\nabla u_{0}\cdot \nabla z&=\int_{D}fz 
&&\forall z\in V_{const}\\
u_0&=g &&\mbox{ on } \partial D
\end{align}

The problem~\eqref{eq:limitproblem} is elliptic and it has a unique solution. 
To analyze this problem further, it is natural to 
define a \emph{harmonic characteristic function} 
$\chi_{D_1}\in H^1_0(D)$ such that
\[
\chi_{D_1}^{(1)}=1 \quad \mbox{ in } D_1,\]
and the harmonic extension of its boundary data in $D_0$ is given by
\begin{align}\label{eq:def:chiD1}
\displaystyle \int_{D_0} \nabla \chi_{D_1}^{(0)} \cdot \nabla z &=0 &&\forall z\in H^1_0(D_0) \nonumber\\ 
\chi_{D_1}^{(0)}&=1 &&\mbox{ on } \partial D_1, \\
\chi_{D_1}^{(0)}&=0 &&\mbox{ on } \partial D. \nonumber
\end{align}

To obtain an explicit formula for $u_0$ we will use the following facts:
$(i)$~problem~\eqref{eq:limitproblem} is elliptic and has a unique solution, 
and $(ii)$~a property of the harmonic characteristic functions described
in the Remark below.
 
\begin{remark}\label{rem:property-of-chi}
Let $w$ be a harmonic extension to $D_0$ of its Neumann data on 
$\partial D_0$. That is, $w$ satisfy the following problem,
\[
\int_{D_0}\nabla w\cdot \nabla z =\int_{\partial D_0}\nabla w\cdot n_0 \,z\quad 
\mbox{ for all } z\in H^1(D_0).
\]
Since $\chi_{D_1}=0$ on $\partial D$ and 
$\chi_{D_1}=1$ on $\partial D_1$, we readily have  that 
\[
\int_{D_0}\nabla \chi_{D_1}\cdot\nabla w=
\int_{\partial D_0} \nabla w\cdot n 
\chi_{D_1}=
0\left(\int_{\partial D} \nabla w\cdot n\right) +
1\left(\int_{\partial D_1} \nabla w\cdot n_0\right) \]
and we conclude that for every harmonic function on $D_0$,  
\begin{equation}\label{eq:def:property-of-chiD1}
\int_{D_0}\nabla \chi_{D_1}\cdot\nabla w=
\int_{\partial D_1} \nabla w\cdot n_0.
\end{equation}
In particular, taking $w=\chi_{D_1}$ we have:
\begin{equation}\label{eq:chi_flux}
\int_{D}|\nabla \chi_{D_1}|^2=
\int_{D_0}|\nabla \chi_{D_1}|^2=\int_{\partial D_1}\nabla 
\chi_{D_1}\cdot n_0.
\end{equation}
Note also that if $\xi\in H^1(D)$ is such that 
$\xi^{(1)}=\xi|_{D_1}=c$ is 
constant in $D_1$ and $\xi^{(0)}=\xi|_{D_0}$ is 
harmonic in $D_0$, then, $\xi=c\chi_{D_1}$.
\end{remark}

We can decompose $u_0$ into the harmonic extension of 
its  constant value in $D_1$, $c(u_0)$,  plus 
the remainder  $u_{0,0}\in H^1(D_0)$. Thus, we write,  
\[
u_{0}=u_{0,0}+c(u_0)\chi_{D_1} 
\]
where $u_{0,0}\in H^1(D)$ is defined by
$u_{0,0}^{(1)}=0$ in $D^{(1)}$ and $u_{0,0}^{(0)}$ solves 
the following Dirichlet  problem,
\begin{align}\label{eq:def:u00}
\displaystyle \int_{D_0} \nabla u_{0,0}^{(0)}\cdot \nabla z &=\int_{D_0}fz &&\forall z\in H^1_0(D_0).\\
u_{0,0}^{(0)}&=0 &&\mbox{ on } \partial D_1,\nonumber\\
u_{0,0}^{(0)}&=g &&\mbox{ on } \partial D.\nonumber
\end{align}
From equation~\eqref{eq:limitproblem} and the observations
in Remark~\ref{rem:property-of-chi} we get that 
\[
\int_{D_0}\nabla (u_{0,0}+c(u_0)\chi_{D_1}) \cdot\nabla 
\chi_{D_1} =\int_{D} f\chi_{D_1}
\]
from which we can obtain
\begin{equation}\label{ref:constantcu0}
c(u_0)=
\frac{\int_{D}f\chi_{D_1} -
\int_{D_0}\nabla u_{0,0}\cdot\nabla \chi_{D_1}
}{\int_{D_0}|\nabla\chi_{D_1}|^2}.
\end{equation}

There is a useful  alternative expression for $c(u_0)$ in~\eqref{ref:constantcu0}
that we also use. 
By using the  
Neumann problem related to $u_{0,0}$  we have that,
\begin{eqnarray*}
\int_{D_0} \nabla u_{0,0}\cdot\nabla \chi_{D_1}&=&
\int_{D_0}f\chi_{D_1}+\int_{\partial D_0}\nabla u_{0,0}\cdot n_0
\chi_{D_1}\\&=&\int_{D_0}f\chi_{D_1}+\int_{\partial D_1}\nabla u_{0,0}\cdot 
n_01
\end{eqnarray*}
and then noting that $\int_{D}f\chi_{D_1}=
\int_{D_0}f\chi_{D_1}+\int_{D_1}f$ we get,
\begin{align}\label{ref:constantcu0-v2}
c(u_0)=
\dfrac{\int_{D_1}f-\int_{\partial D_1}\nabla u_{0,0}\cdot n_0
}{\int_{\partial D_1}\nabla\chi_{D_1}\cdot n_0},
\end{align}
which reveals that $c(u_0)$ balances the fluxes across $\partial D_1$. To summarize the results obtained to this point, we can express $u_0$ as follows:
\begin{align}
u_{0}&=u_{0,0}
   +\dfrac{\int_{D}f\chi_{D_1} -\int_{D_0}\nabla u_{0,0}\cdot \nabla \chi_{D_1}}
{\int_{D_0}|\nabla\chi_{D_1}|^2} \chi_{D_1}
\label{eq:u0-equal-u00-plus-chiD1}\\
&=u_{0,0}+{
\dfrac{\int_{D_1}f-\int_{\partial D_1}\nabla u_{0,0}\cdot n_0}
{\int_{\partial D_1}\nabla\chi_{D_1}\cdot n_0}}\chi_{D_1}
\label{eq:u0-equal-u00-plus-chiD1v2}
\end{align}

Given the explicit form of $u_0$, we use it in~\eqref{eq:eta0} 
to find $u_1^{(1)}=u_1|_{D_1}$, from if we conclude that 
 $u_0^{(0)}$ and $u_1^{(1)}$ satisfy the 
local Dirichlet problems
\begin{align*}
\int_{D_0}\nabla u_0\cdot \nabla z&=\int_{D_0}fz &&\forall z\in H^1_0(D_0),\\
\int_{D_1}\nabla u_1\cdot \nabla z&=\int_{D_1}fz &&\forall z\in H^1_0(D_1),
\end{align*}
with given boundary data on $\partial D_0$ and $\partial D_1$.
Equation~\eqref{eq:eta0}  also represents the transmission 
conditions across $\partial D_1$ for the functions $u_{0}^{(0)}$ 
and $u_1^{(1)}$. This is easier to see when  
the forcing $f$ is square integrable.
From now on, in order to 
simplify the presentation, we assume that 
$f\in L^2(D)$. If $f\in L^2(D)$,  we have that 
$u_0^{(0)}$ and $u_1^{(1)}$ are the only solutions of the  
problems:
\[
\int_{D_0}\nabla u_0^{(0)}\cdot \nabla z=\int_{D_0}fz +
\int_{\partial{D_0}\setminus \partial D}\nabla u_0^{(0)}\cdot
n_0 z  \quad
\forall z\in H^1({D_0}) 
\mbox{ with } z=0 \mbox{ on } \partial D\] 
with $u_0^{(0)}={g}$ on $\partial D$, and 
\[
\int_{D_1}\nabla u_1^{(1)}\cdot \nabla z=\int_{D_1}fz +
\int_{\partial D_1}\nabla u_1^{(1)}\cdot n_1  z \quad \mbox{ for all } z\in H^1(D_1).
\]
Replacing these  last two equations back into~\eqref{eq:eta0} 
we conclude that 
\[
\int_{\partial D_1} (\nabla u_0^{(0)}\cdot n_0+\nabla u_1^{(1)}\cdot n_1)z=0
\quad \mbox{ for all } z\in H^1(D),\]
which implies
\[
\nabla u_1^{(1)}\cdot n_1=-\nabla u_0^{(0)}\cdot n_0 \quad \mbox{ on }
\partial D_1.
\]
Using this interface condition we can obtain $u_1^{(1)}$ in $D_1$ 
by writing 
\[
u_1^{(1)}=\widetilde{u}_1^{(1)}+c_1 \quad \mbox{ where } \int_{D_1}\widetilde{u}_1^{(1)}=0
\]
and $\widetilde{u}_1^{(1)}$ solves the Neumann problem
\begin{equation}\label{eq:Neumanu1}
\int_{D_1}\nabla \widetilde{u}_1^{(1)}\cdot \nabla z=\int_{D_1}fz -
\int_{\partial D_1}\nabla u_0^{(0)}\cdot n_1  z
\quad \mbox{ for all } z\in H^1(D_1).
\end{equation}
The constant $c_1$ will be chosen later.  
Problem~\eqref{eq:Neumanu1} satisfies 
the compatibility condition, 
\begin{eqnarray*}
\int_{\partial D_1}\nabla \widetilde{u}_1^{(1)}\cdot n_1&=&-
\int_{\partial D_1 }\nabla u_0^{(0)}\cdot n_0\\&=&
-\int_{\partial D_1}\nabla u_0^{(0)}\cdot n_0-c(u_0)\int_{\partial D}
\nabla \chi_{D_1}\cdot n_0
=\int_{D_1}f.
\end{eqnarray*}
Here we use the value of $c(u_0)$ is given in~\eqref{ref:constantcu0-v2}.

Next, we discuss how to compute $u_1^{(0)}$ and 
$\widetilde{u}_1^{(0)}$ to  completely define 
the functions $u_1\in H^1(D)$ and $\widetilde{u}_1\in H^1(D)$.
These are presented for general 
$i\geq 1$ since the construction is independent of $i$ in this range.

\subsubsection{Term corresponding to $\eta^{-i}$ with $i\geq 1$:}
For powers of ${1}/{\eta}$ larger or equal to one there are only two 
terms in the summation that lead to the following system:
\begin{equation}\label{eq:term-i}
 \int_{D_0}\nabla u_{i}\cdot \nabla v+\int_{D_1}\nabla u_{i+1}\cdot \nabla v=0
\quad \mbox{ for all } v\in H^1_0(D).
\end{equation}
This equation represents both the subdomain problems 
and the transmission conditions across $\partial D_1$ 
for $u_i^{(0)}$ 
and $u_{i+1}^{(1)}$. 
Following a similar argument to the one given above, 
we conclude that $u_i^{(0)}$ is harmonic in $D_0$ for all $i\geq 1$ and that 
$u_i^{(1)}$ is harmonic in $D_1$ for $i\geq 2$. As before, we
have 
\begin{equation}\label{eq:neumann-cond-generali}
\nabla u_{i+1}^{(1)}\cdot n_1=-\nabla u_i^{(0)}\cdot n_0.
\end{equation}

We note that $u_i^{(1)}$ in $D_1$, (e.g., $u_1^{(1)}$ above) 
is given by 
the solution of a Neumann problem in $D_1$. 
To uniquely determine $u_i^{(1)}$, we impose the 
condition $\int_{D_1}u_i^{(1)}=c_i$ and write 
\begin{equation}
u_i^{(1)}=\widetilde{u}_i^{(1)}+c_i \quad \mbox{ where } \int_{D_1}
\widetilde{u}_i=0.
\end{equation}
where the appropriate  $c_i$ will be determined later.

Given $u_i^{(1)}$ in $D_1$ we find $u_i^{(0)}$ in $D_0$ by solving
a Dirichlet problem with known Dirichlet data, that is, 
\begin{equation}\label{eq:DirforuiD0}
\begin{array}{l}
\displaystyle \int_{D_0}\nabla u_{i}^{(0)}\cdot \nabla z=0 \mbox{ for all } 
z\in H^1_0(D_0) \\\\
u_i^{(0)}=u_i^{(1)} ~~(=\widetilde{u}_i^{(1)}+c_i) \mbox{ on } 
\partial D_1
\quad \mbox{ and } \quad u_i=0 \mbox{ on } \partial D.
\end{array} 
\end{equation}

Since $c_i$, $i=1,\dots,$  are constants, their harmonic extensions 
are  given by $c_i\chi_{D_1}$, $i=1,\dots$; see
Remark~\ref{rem:property-of-chi}. Then,  we conclude that 
\begin{equation}
u_i=\widetilde{u}_i+c_i\chi_{D_1}
\end{equation}
where $\widetilde{u}_i^{(0)}$ is defined by~\eqref{eq:DirforuiD0}
 replacing $c_i$ by $0$. This completes the 
construction of $u_{i}$. 

Now we proceed to show how to 
to find 
$u_{i+1}^{(1)}$ in $D_1$. For this,  we use~\eqref{eq:term-i} 
and~\eqref{eq:neumann-cond-generali} 
which lead to the following Neumann problem
\begin{equation}\label{eq:Neuforuiplu1D1}
 \int_{D_1}\nabla \widetilde{u}_{i+1}^{(1)}\cdot \nabla z=
-\int_{\partial D_1 }\nabla 
u_{i}^{(0)}\cdot n_0z \quad \mbox{ for all } 
z\in H^1(D_1).
\end{equation}

The compatibility condition for this Neumann problem 
is satisfied if we choose 
\begin{equation}\label{eq:def:ci}
c_i=-\frac{\int_{\partial D_1} \nabla \widetilde{u}_{i}^{(0)}\cdot n_0  }{
\int_{\partial D_1} \nabla \chi_{D_1}^{(0)}\cdot n_0}=
-\frac{\int_{D_0}\nabla \widetilde{u}_i \cdot\nabla \chi_{D_1} }{
\int_{D_0} |\nabla \chi_{D_1}|^2}.
\end{equation}
For the second equality see Remark~\ref{rem:property-of-chi} below 
and Equations~\eqref{ref:constantcu0} and~\eqref{ref:constantcu0-v2}.
The compatibility conditions trivially satisfy
\begin{align*}
\int_{\partial D_1}\nabla \widetilde{u}_{i+1}^{(1)}\cdot n_1&=
-\int_{\partial D_1}\nabla u_{i}^{(0)}\cdot n_0\\&=
-\int_{\partial D_1}\nabla(\widetilde{u}_{i}^{(0)}
+c_i
\chi_{D_1}^{(0)})\cdot n_0\\
&=
-\int_{\partial D_1} \nabla \widetilde{u}_{i}^{(0)}\cdot n_0-c_i
\int_{\partial D_1}\nabla \chi_{D_1}^{(0)}\cdot n_0\\&=0.
\end{align*}
where we have used the definitions of $c_i$ given in~\eqref{eq:def:ci}.

We can choose $u_{i+1}^{(1)}$ in $D_1$ such that 
\[
u_i^{(1)}=\widetilde{u}_{i+1}^{(1)}+c_{i+1} \quad 
\mbox{ where } \int_{D_1}\widetilde{u}_{i+1}=0,
\]
and, as before, 
\[
c_{i+1}=-\frac{\int_{\partial D_1} \nabla \widetilde{u}_{i+1}^{(0)}
\cdot n_0  }{
\int_{\partial D_1} \nabla \chi_{D_1}^{(0)}\cdot n_0}=
-\frac{\int_{D_0}\nabla \widetilde{u}_{i+1} \cdot\nabla \chi_{D_1} }{
\int_{D_0} |\nabla \chi_{D_1}|^2},
\]
so we have the compatibility condition 
of the Neumann problem to compute $u_{i+2}^{(1)}$.
See the Equation~\eqref{eq:Neuforuiplu1D1}.

\subsubsection{Summary}\label{sec:summary}
We summarize the Dirichlet-to-Neumann 
procedure to compute the terms of the  asymptotic expansion 
for $u_\eta$ in~\eqref{eq:expansion}-\eqref{eq:expansionBC}.
\begin{enumerate}
\item Compute $u_{0}$ using formulae~\eqref{eq:u0-equal-u00-plus-chiD1} 
or~\eqref{eq:u0-equal-u00-plus-chiD1v2}.
\item Compute $u_1^{(1)}$ in $D_1$ by solving 
the Neumann problem~\eqref{eq:Neumanu1}. 
Compute $u_1^{(0)}$ in $D_0$ solving 
the Dirichlet problem~\eqref{eq:DirforuiD0} with $i=1$.
\item For $i=2,3,\dots$ compute $u_i^{(1)}$ in $D_1$ by solving the 
Neumann 
problem~\eqref{eq:Neuforuiplu1D1}. Then,  compute 
$u_i^{(0)}$ in $D_0$ solving 
the Dirichlet problem~\eqref{eq:DirforuiD0}.

\end{enumerate}

Other cases can be considered. For instance, an expansion 
for the case where we interchange 
$D_0$ and $D_1$ can also be analyzed. In this case
the asymptotic solution is not constant in 
the high-conducting part. 
Multiple inclusions will be consider in Section~\ref{sec:high-multiple}. 

\subsection{Convergence in $H^1(D)$}\label{sec:convergence}
In  this section we study the convergence 
of the expansion~\eqref{eq:expansion}-\eqref{eq:expansionBC}.
For simplicity 
of the presentation we consider the case of one high-conductivity 
inclusion. The converge results will be extended to 
the multiple high-conductivity inclusions in 
Section~\ref{sec:high-multiple}.
We assume that $\partial D$ and $\partial 
D_1$ are sufficiently smooth, see~\cite{Grisvard}. 
\begin{lemma}\label{lemma:firstbounds}
Let $u_0$  in~\eqref{eq:u0-equal-u00-plus-chiD1}, 
with $u_{0,0}$ defined in~\eqref{eq:def:u00}, 
and  $u_1$ be defined by~\eqref{eq:Neumanu1} and~\eqref{eq:DirforuiD0}
with $i=1$. 
We have that, 
\begin{equation}\label{eq:bound:u0-by-f-and-g}
|u_0|_{H^1(D)}
\preceq   \|f\|_{H^{-1}(D)}+\|g\|_{H^{1/2}(\partial D)},
\end{equation}
\begin{eqnarray}\label{eq:boundu1tilde-by-fandg}
\|\widetilde{u}_1\|_{H^1(D_1)} &\preceq& \|f\|_{H^{-1}(D_1)}
+\|g\|_{H^{1/2}(D)}\end{eqnarray}
and 
\begin{equation}\label{eq:boundu1tildeD0-by-u1tildeD1}
\|\widetilde{u}_{1}\|_{H^1(D_0)}\preceq
\|\widetilde{u}_{1}\|_{H^{1/2}(\partial D_1)}\preceq
\|\widetilde{u}_{1}\|_{H^1(D_1)}.
\end{equation}

\end{lemma}
\begin{proof}
From the definition of $u_{0,0}$ in~\eqref{eq:def:u00} we have that 
\[
\|u_{0,0}\|_{H^1(D_0)}\preceq \|f\|_{H^{-1}(D_0)}+
\|g\|_{H^{1/2}(\partial D)}.
\]
Using~\eqref{eq:u0-equal-u00-plus-chiD1} we have that
\begin{eqnarray}
|u_0|_{H^1(D)}&=&
|u_0|_{H^1(D_0)}\leq |u_{0,0}|_{H^1(D_0)}+|c(u_0)||\chi_{D_1}|_{H^1(D_0)}
\end{eqnarray}
and we observe that 
\begin{eqnarray*}
|c(u_0)||\chi_{D_1}|_{H^1(D_0)}
&\preceq& 
|\frac{\int_{D}f\chi_{D_1} -
\int_{D_0}\nabla u_{0,0}\cdot\nabla \chi_{D_1}}{|\chi_{D_1}|_{H^1(D_0)}}|\\
&\preceq& \|f\|_{H^{-1}(D)}+\|g\|_{H^{1/2}(\partial D)}.
\end{eqnarray*}
This proves~\eqref{eq:bound:u0-by-f-and-g}.
Equation~\eqref{eq:boundu1tilde-by-fandg} follows 
from the classical estimate for problem~\eqref{eq:Neumanu1}. 
Equation~\eqref{eq:boundu1tildeD0-by-u1tildeD1} follows 
from problem~\eqref{eq:DirforuiD0} with $i=1$ and  a 
trace theorem;  see~\cite{Grisvard}.\end{proof}

The following lemma can be obtained using orthogonality 
relations of Galerkin projections.
\begin{lemma}\label{lem:bound-for-c(w)}
If $\widetilde{w}\in H^1(D)$, $\widetilde{w}$ is harmonic in $D_0$  and we define 
\[
w=\widetilde{w}+c(w)\chi_{D_1}\]
where 
\[
c(w)=-\frac{\int_{\partial D_1} \nabla \widetilde{w}\cdot n_0  }{
\int_{\partial D_1} \nabla \chi_{D_1}\cdot n_0}=
-\frac{\int_{D_0}\nabla \widetilde{w} \cdot\nabla \chi_{D_1} }{
\int_{D_0} |\nabla \chi_{D_1}|^2},
\]
then, $w$ and $\chi_{D_1}$ are orthogonal in 
the operator norm induced by the Dirichlet operator, that is,  
$\int_{D}\nabla w\nabla \chi_{D_1}=0$.  We also have 
$|\widetilde{w}|^2_{H^1(D)}=|w|^2_{H^1(D)}+c(w)^2|\chi_{D_1}|^2_{H^1(D)}$, 
\[
|w|_{H^1(D)}\leq |\widetilde{w}|_{H^1(D)} \mbox{ and } 
\|w\|_{H^1(D)}\preceq  \|\widetilde{w}\|_{H^1(D)}.
\]
Here, the hidden constant is the 
  Poincar\'e-Friedrichs  inequality constant on $D$.
\end{lemma}
The next lemma bound the norm of the $i-$th term by 
the norm of the $(i-1)-$th in the asymptotic 
expansion~\eqref{eq:expansion}.
\begin{lemma}\label{lem:boud-uiplus1-by-ui}
Let $\widetilde{u}_i$ defined on $D_0$ by~\eqref{eq:DirforuiD0} 
with $c_{i}=0$, 
and $u_{i+1}$ defined on $D_1$ by~\eqref{eq:Neuforuiplu1D1}.
For $i\geq 1$ we have that
\[
\|u_{i+1}\|_{H^1(D)}\preceq 
\|\widetilde{u}_{i}\|_{H^1(D_0)}.
\]
\end{lemma}
\begin{proof}
Let $i\geq 1$. 
Consider $\widetilde{u}_{i+1}$ defined by 
the Dirichlet problem~\eqref{eq:DirforuiD0}. 
From classical estimates of the solution on $D_0$ 
and the trace theorem on $D_1$, we have 
\[
\|\widetilde{u}_{i+1}\|_{H^1(D_0)}\preceq
\|\widetilde{u}_{i+1}\|_{H^{1/2}(\partial D_1)}\preceq
\|\widetilde{u}_{i+1}\|_{H^1(D_1)}.
\]
By considering the problem~\eqref{eq:Neuforuiplu1D1}
we conclude that 
\[
\|\widetilde{u}_{i+1}\|_{H^1(D_1)}\preceq 
\|{u}_{i}\|_{H^1(D_0)}.
\]

We have, form~\eqref{eq:def:ci}
and Lemma~\ref{lem:bound-for-c(w)}, we have
\[
\|{u}_{i+1}\|_{H^1(D)}\preceq 
\|\widetilde{u}_{i+1}\|_{H^1(D)}.
\]
Combining this last three inequalities we have 
\[
\|{u}_{i+1}\|_{H^1(D)}\preceq 
\|\widetilde{u}_{i+1}\|_{H^1(D)}\preceq \|u_{i}\|_{H^1(D_0)}.
\]
The  constants are independent of $i$ and depend 
only on the domain geometry and configuration, 
that is, on $D_1$ and $D_0$. In fact, 
the hidden constants depend on the trace theorem and 
solution estimates in $D_1$ and $D_0$,  see~\cite{Grisvard}.

\end{proof}

\begin{theorem}\label{thm:convergence-one-high}
There is a constant $C>0$ such that for every 
$\eta>C$, the expansion~\eqref{eq:expansion} converges 
(absolutely) in $H^1(D)$. The asymptotic limit 
$u_0$ satisfies problem~\eqref{eq:limitproblem} and $u_0$ 
can be computed using formula~\eqref{eq:u0-equal-u00-plus-chiD1}.
\end{theorem}
\begin{proof}
From Lemma~\ref{lem:boud-uiplus1-by-ui} applied 
repeatedly $i-1$ times, we get that  for every $i\geq 2$ 
 there is a constant $C$ such that
\begin{eqnarray*}
\|u_{i}\|_{H^1(D)}&\leq& C\|u_{i-1}\|_{H^1(D_0)}\leq
C\|u_{i-1}\|_{H^1(D)} \\
&\leq & \dots\\
&\leq& C^{i-1}\|\widetilde{u}_1\|_{H^1(D_0)}
\end{eqnarray*}
and then 
\[
\|\sum_{i=2}^\infty \eta^{-i}u_{i}\|_{H^1(D)}\leq 
\frac{\|\widetilde{u}_1\|_{H^1(D_0)} }{C}\sum_{i=2}^\infty\left(\frac{C}{\eta}\right)^{i}.
\]
The last expansion converges when $\eta>C$. 
Using~\eqref{eq:bound:u0-by-f-and-g} 
and~\eqref{eq:boundu1tilde-by-fandg} we conclude there 
is a constant $C_1$ such that
that \[\|\sum_{i=0}^\infty \eta^{-i}u_{i}\|_{H^1(D)} 
\preceq C_1 (\|f\|_{H^{-1}(D_1)}
+\|g\|_{H^{1/2}(D)}) \sum_{i=0}^\infty\left(\frac{C}{\eta}\right)^{i}.\]
Moreover, 
the asymptotic limit $u_0$ satisfies problem~\eqref{eq:limitproblem}.
\end{proof}

\begin{corollary}
There are positive  constants $C$ and $C_1$ such that for every 
$\eta>C$, we have 
\[
\|u-\sum_{i=0}^I \eta^{-i}u_{i}\|_{H^1(D)}\leq C_1
(\|f\|_{H^{-1}(D_1)}
+\|g\|_{H^{1/2}(D)})\sum_{i=I+1}^\infty\left(\frac{C}{\eta}\right)^{i},
\]
for $I\geq 0$.
\end{corollary}

We note that in the case of smooth boundaries 
$\partial D_1$, $\partial D$ and smooth 
Dirichlet data and forcing term,  we have
$H^{1+s}(D_1)$  and $H^{1+s}(D_0)$ regularity 
of all functions involved for $s>0$; see 
\cite{cgh09, Grisvard} and references therein. Estimates similar 
to the ones presented in this section will 
warrant that for $\eta$ sufficiently large, 
the expansion~\eqref{eq:expansion}-\eqref{eq:expansionBC}
will be absolutely converging in $H^{1+s}(D_1)$ and 
$H^{1+s}(D_0)$ for $\eta$ sufficiently large. 
A more delicate case is the case with non-smooth 
boundaries. This case and the convergence of 
the expansion in $H^{1+\tau}(D)$ for some small $\tau>0$ will 
object of future research.

\subsection{Multiple high-conductivity inclusions}\label{sec:high-multiple}
In this section we consider a coefficient with 
multiple high-conductivity inclusions. Let $\kappa$ be defined by 
\begin{equation}\label{eq:coeff1inc-multiple}
\kappa(x)=\left\{\begin{array}{cc} 
\eta,& x\in   D_m, ~~m=1,\dots,M, \\
1,& x\in  D_0=D\setminus \cup_{m=1}^M\overline{D}_m,
\end{array}\right.
\end{equation}
and denote by $u_{\eta}$
 the solution of~\eqref{eq:problem}
with zero Dirichlet boundary condition.
 We assume that 
$D_i$ is compactly included in the open set $D\setminus 
\cup_{\ell=1, \ell \not=m}^M\overline{D}_\ell$, i.e., 
$\overline{D}_m\subset D\setminus \cup_{\ell=1, \ell\not=m}^M \overline{D}_\ell$, and we define $D_0:= D\setminus 
\cup_{m=1}^M\overline{D}_m$.

Expansion~\eqref{eq:expansion}-\eqref{eq:expansionBC} holds 
in this case. We first describe the asymptotic problem 
in the next Section~\ref{sec:asymp-problem-multiple}. 
Then we will  quickly describe the expansion
in Section~\ref{sec:expansions-description} below.

\subsubsection{The solution of the asymptotic problem}\label{sec:asymp-problem-multiple}
Define the set of constant functions inside the inclusions, 
\[
V_{const}=\{ v\in H^1_0(D), \mbox{ such that } v|_{D_m} 
\mbox{ is  constant for all } m=1,\ldots,\, M\}.
\]
By analogy with the 
case of one high-conductivity inclusion, the asymptotic solution 
for the coefficient~\eqref{eq:coeff1inc-multiple} is $u_0$ 
that is constant in each high-conductivity inclusions. Moreover, 
$u_0$ solves the problem, 
\begin{equation}\label{eq:limitproblemmultiple}
\begin{array}{c}
 \int_{D_0}\nabla u_{0}\cdot \nabla z=\int_{D}fz 
\mbox{ for all } z\in V_{const},\\\\
u_0=g \mbox{ on } \partial D.  
\end{array}
\end{equation}

The problem above is elliptic and it has a unique solution. 
For $m=1,\dots,M$, define the  harmonic characteristic function
$\chi_{D_m}\in H^1_0(D)$ by 
\[
 \quad 
\chi_{D_m}\equiv\delta_{ m\ell} \mbox{ on } D_\ell \quad  \text{ for } \ell=1,\ldots,\, M,
\]
and, in $D_0$, $\chi_{D_m}$  is defined as the harmonic extension of its boundary data in $D_0$, i.e.,
\begin{equation}\label{eq:def:chiDl}
\begin{array}{rl}
\displaystyle \int_{D_0} \nabla \chi_{D_m} \cdot \nabla z =&0 \quad\mbox{ for all } z\in H^1_0(D_0)  \\
\chi_{D_m}=&\delta_{m\ell }\mbox{ on } \partial D_\ell \mbox{ for }
\ell=1,\dots,M,\\
\chi_{D_m}=&0 \mbox{ on } \partial D.
\end{array}
\end{equation}
Here, $\delta_{m\ell}$ represent the Kronecker delta, which is equal to 1 when $m=l$ and 0 otherwise.
Remark~\ref{rem:property-of-chi} 
holds if we replace the one inclusion case 
$\chi_{D_1}$ with the multi-inclusion case 
$\chi_{D_\ell}$ defined in~\eqref{eq:def:chiDl}. For instance,  if $w\in H^1(D)$  is harmonic in 
$D_0$ and constant $w=c_m$ in $D_m$, $m=1,\dots,M$, then, 
we can write $w=\sum_{m=1}^M c_m\chi_{D_m}$. 

We decompose $u_0$ into the 
 harmonic extension (to $D_0$) of a 
 function in $V_{const}$, plus a  function, $u_{0,0}$, 
with $g$  boundary 
condition on $\partial D$ and zero boundary condition 
on $\partial D_m$, $m=1,\dots,M$. We write,    
\begin{equation}\label{eq:linearcomb-multipleinc}
u_{0}=u_{0,0}+
\sum_{m=1}^M c_m(u_0)
\chi_{D_m}, 
\end{equation}
where $u_{0,0}\in H^1(D)$ with   
$u_{0,0}=0$ in $D_m$, $m=1,\dots,M,$ and 
$u_{0,0}$ solves the following problem in $D_0$,  
\begin{equation}\label{eq:def:u00-multipleinc}
\begin{array}{c}
\int_{D_0} \nabla u_{0,0}\cdot \nabla z =\int_{D_0}fz \mbox{ for all } z\in H^1_0(D_0),\\\\
u_{0,0}=0 \mbox{ on }  \partial D_m, \quad m=1,\dots,M,
 \quad \mbox{ and } 
u_{0,0}=g \mbox{ on } \partial D.
\end{array}
\end{equation}
Equation~\eqref{eq:linearcomb-multipleinc} is the analogous 
to Equation~\eqref{eq:u0-equal-u00-plus-chiD1}. 
Now we show how to compute the constants $c_i(u_0)$
using the same procedure as before.  From~\eqref{eq:limitproblem}, 
we have
\[
\int_{D_0}\nabla (u_{0,0}+
\sum_{m=1}^M c_m(u_0)
\chi_{D_m}
)\cdot \nabla \chi_{D_\ell} =\int_{D} f\chi_{D_\ell},\quad \text{for }\ell=1,\ldots,\, M,
\]
which is equivalent to the $M\times M$ linear system, 
\begin{equation}
\label{eq:Ageom}
A_{geom} X=B
\end{equation}
where $A=[a_{ij}]$, 
and $B=(b_1,\dots,b_M)\in\mathbb{R}^M$ are defined by 
\begin{equation}\label{eq:def:aij}
a_{ij}=\int_{D}\nabla \chi_{D_i}\nabla \chi_{D_j}=
\int_{D_0}\nabla \chi_{D_i}\nabla \chi_{D_j}, 
\end{equation}
\[
b_j=\int_{D}f\chi_{D_j} - \int_{D_0}\nabla u_{0,0}  \nabla \chi_{D_j}
\]
and $X=(c_1(u_0),\dots,c_M(u_{0}))\in\mathbb{R}^M$. We conclude that 
\begin{equation}\label{ref:vectorc-uinfty}
X=A_{geom}^{-1}B.
\end{equation}

We note that using $(\ref{eq:def:property-of-chiD1})$
for $\chi_{D_i}$ we have that 
\begin{equation}\label{eq:propertyaij}
a_{ij}=\int_{D}\nabla \chi_{D_i}\nabla \chi_{D_j}=
\int_{\partial D_i}\nabla \chi_{D_j}\cdot n_i=
\int_{\partial D_j}\nabla \chi_{D_j}\cdot n_j.
\end{equation}

Note that $\sum_{m=1}^Mc_m(u_0)\chi_{D_m}$ is the solution
of a Galerkin  projection in 
the space $\mbox{Span}\{\chi_{D_m}\}_{m=1}^M$.
The  forcing term for this problem is 
$f$ and there is  Neumann boundary  data  on $\partial D_m$ coming form 
$\nabla u_{0,0}\cdot n_0$. 
 Matrix $A_{geom}$ encodes the 
geometry information concerning the distribution 
of the inclusions inside the domain $D$, while it is 
independent of the contrast $\eta$. Note that in general $A_{geom}$
can be a large dense matrix. Because $\chi_{D_j}$ decay, one can 
approximate the system by a sparser system (e.g., 
see \cite{BerlyandNovikov,yulya-berlyan2,bp98}). Moreover, we can use
concepts similar to multiscale finite element methods and seek smaller
dimensional approximations for this large system.

\subsubsection{Expansion}\label{sec:expansions-description}
Now we describe how to compute the 
individual terms of the asymptotic 
expansion~\eqref{eq:expansion}-\eqref{eq:expansionBC}
for the case of multiple high-conductivity 
inclusions. 
\begin{itemize}
\item The function $u_0$ solves~\eqref{eq:limitproblemmultiple}.
\item The restriction of $u_1$ to the subdomain $D_m$, 
$u_1^{(m)}$,   can be written 
\[
u_1^{(m)}=\widetilde{u}_1^{(m)}+c_{1,m} \mbox{ where } 
\int_{D_m}\widetilde{u}^{(m)}=0,
\]
and $\widetilde{u}_1^{(m)}$ satisfies the Neumann problem,
\begin{equation}\label{eq:Neumanu1-multipleinc}
\int_{D_m}\nabla \widetilde{u}_1^{(m)}\cdot \nabla z=\int_{D_m}fz -
\int_{\partial D_m}\nabla u_0^{(0)}\cdot n_m  z
\quad \mbox{ for all } z\in H^1(D_m),
\end{equation}
for $m=1,\dots,M$. 
The constants $c_{1,m}$, $m=1,\dots,M,$ will be chosen later. 
\item For $i=1,2,\dots$, we have that 
given $u_i^{(m)}$ in $D_m$, 
$m=1,\dots,M$,  we can find $u_i^{(0)}$ in $D_0$ by  solving
 the Dirichlet problem
\begin{equation}\label{eq:DirforuiD0-multiple}
\begin{array}{l}
\displaystyle \int_{D_0}\nabla u_{i}^{(0)}\cdot \nabla z=0 \mbox{ for all } 
z\in H^1_0(D_0)\\\\
u_i^{(0)}=u_i^{(m)} ~~(=\widetilde{u}_i^{(m)}+c_{i,m})
\mbox{ on } \partial D_m, \quad m=1,\dots,M,
\quad \mbox{ and } \\ u_i^{(0)}=0 \mbox{ on } \partial D.
\end{array} 
\end{equation}

Since $c_{i,m}$ are  constants, the corresponding 
 harmonic extension 
is given by $\sum_{m} c_{i,m}\chi_{D_\ell}$. 
Then,  we conclude that 
\begin{equation}
u_i=\widetilde{u}_i+\sum_{m}c_{i,m}\chi_{D_m}
\end{equation}
where $\widetilde{u}_i^{(0)}$ is defined by~\eqref{eq:DirforuiD0-multiple}
 replacing all the constants $c_{i,m}$ by $0$. 

The $u_{i+1}^{(m )}$ in $D_m$ satisfy the following Neumann problem
\begin{equation}\label{eq:Neuforuiplu1D1-multiple}
 \int_{D_m}\nabla u_{i+1}^{(m)}\cdot \nabla z=-\int_{\partial D_m }
\nabla 
u_{i}^{(0)}\cdot n_0z \quad \mbox{ for all } 
z\in H^1(D_\ell).
\end{equation}

For the compatibility condition we need that for $\ell=1,\dots,M$, 
\begin{eqnarray*}
0=\int_{\partial D_\ell }\nabla u_{i+1}^{(\ell)}\cdot n_\ell&=&
-\int_{\partial D_\ell}\nabla u_{i}^{(0)}\cdot n_0\\&=&
-\int_{\partial D_\ell}\nabla(\widetilde{u}_{i}^{(0)}
+\sum_{m=1}^Mc_{i,m}
\chi_{D_m}^{(0)})\cdot n_0\\
&=&
-\int_{\partial D_\ell} \nabla \widetilde{u}_{i}^{(0)}\cdot n_0-
\sum_{m=1}^Mc_{i,m}\int_{\partial D_\ell}\nabla \chi_{D_m}^{(0)}\cdot n_0.\end{eqnarray*}
From~\eqref{eq:def:aij} and~\eqref{eq:propertyaij} 
we conclude that 
$Y_i=(c_{i,1},\dots,c_{i,M})$ is the solution of 
\begin{equation}\label{eq:def:Yi}
A_{geom} Y_i = {U}_i
\end{equation}
where 
\[
{U}=(-\int_{\partial D_{1}}
\nabla \widetilde{u}_i^{(0)}\cdot n_0, \dots, 
-\int_{\partial D_{m}}\widetilde{u}_i^{(0)}\cdot n_0).
\]
or (using Remark~\ref{rem:property-of-chi}), 
\[
{U}=(-\int_{D_0}
\nabla \widetilde{u}_i^{(0)}\nabla \chi_{D_1}, \dots, 
-\int_{ D_0}\nabla\widetilde{u}_i^{(0)}\nabla \chi_{D_M}).
\]

\end{itemize}
\subsubsection{Convergence in $H^1(D)$}
We first prove the result analogous to Lemma~\ref{lem:bound-for-c(w)}.
\begin{lemma}\label{lem:bound-wtilde-by-w}
Let $\widetilde{w}\in H^1(D)$ be harmonic 
in $D_0$ and define
$
w=\widetilde{w}+\sum_{m=1}^M c_m\chi_{D_m},
$
where $Y=(c_1,\dots,c_M)$ is the solution of the $M$ 
dimensional linear system 
\[
A_{geom}Y=-W
\]
with $W=(\int_{D_0}\nabla w\nabla \chi_{D_1}, 
\dots,\int_{D_0}\nabla w\nabla \chi_{D_m})$.
Then, 
\[
\|w\|_{H^1(D)}\preceq \|\widetilde{w}\|_{H^1(D)}
\]
where the hidden constant is the  Poincar\'e-Friedrichs  inequality 
constant of $D$.
\end{lemma}

\begin{proof}
Note that $\sum_{m}^M c_m\chi_{D_m}$ is the 
Galerkin projection of $\widetilde{w}$ into the space 
$\mbox{span}\{\chi_{D_i}\}_{i=1}^M$. Then, 
as usual in Finite Element analysis of Galerkin formulations, we have
\begin{align*}
\displaystyle\int_{D_0} |\sum_{m=1}^M c_i \nabla \chi_{D_m}|^2 & = 
Y^TA_{geom}Y =-Y^TW\\
&=-\sum_{m=1}^M c_m \int_{D_0}\nabla w \nabla \chi_{{D_m}} \\
&=-\int_{D_0}\nabla w \nabla (\sum_{m=1}^M c_m\chi_{D_m})\\
&\leq|w|_{H^1(D_0)}\left|\sum_{m=1}^M c_m\chi_{D_m}\right|_{D_0}
\end{align*}
and then $\left|\sum_{m=1}^M c_m\chi_{D_m}\right|\leq |w|_{H^1(D_0)}$. 
Using a  Poincar\'e-Friedrichs  inequality we can write,
\[
\|w\|_{H^1(D)}\leq \|\widetilde{w}\|_{H^1(D)}+
\left\|\sum_{m=1}^M c_m\chi_{D_m}\right\|\preceq \|\widetilde{w}\|_{H^1(D)}.
\]\end{proof}

Combining Lemma~\ref{lem:bound-wtilde-by-w} with 
 results analogous to Lemmas~\ref{lemma:firstbounds} 
and~\ref{lem:boud-uiplus1-by-ui} 
we get convergence for the 
expansion~\eqref{eq:expansion}-\eqref{eq:expansionBC}.
\begin{theorem}\label{thm:convergence-multiple-high}
Consider the problem~\eqref{eq:problem} 
with coefficient~\eqref{eq:coeff1inc-multiple}.  The 
corresponding expansion~\eqref{eq:expansion}-\eqref{eq:expansionBC}
converges absolutely in $H^1(D)$ for $\eta$ sufficiently large.  
Moreover, there exist positive constants $C$ and $C_1$ such that for every 
$\eta>C$, we have 
\[
\|u-\sum_{i=0}^I \eta^{-i}u_{i}\|_{H^1(D)}\leq 
C_1\left(\|f\|_{H^{-1}(D_1)}
+\|g\|_{H^{1/2}(D)} \right)\sum_{i=I+1}^\infty\left(\frac{C}{\eta}\right)^{i},
\]
for $I\geq 0$.
\end{theorem}

\section{The case of low-conductivity inclusions}\label{sec:low}

In this section we derive and analyze expansions 
for the case of low-conductivity inclusions. 
As before, we present the case of one single inclusion 
first (see Section~\ref{small-sec:one}) and analyze the general case in Section~\ref{small-sec:multiple}.
\subsection{Expansion  derivation: one low-conductivity inclusion}
\label{small-sec:one}

Let $\kappa$ be defined by 
\begin{equation}\label{small-eq:coeff1inc}
\kappa(x)=\left\{\begin{array}{ll} 
\epsilon,& x\in   D_1,  \\
1,\quad& x\in  D_0=D\setminus \overline{D}_1,
\end{array}\right.
\end{equation}
and denote by $u_{\epsilon}$ the solution of~\eqref{eq:problem}.
 We  assume that 
$D_1$ is compactly included in $D$ ($\overline{D}_1\subset D$).
Since $u_\epsilon$ is solution of~\eqref{eq:problem} with 
the coefficient~\eqref{small-eq:coeff1inc} we have
\begin{equation}\label{small-eq:problem-with-1-inclusion}
\int_{D_0}\nabla u_{\epsilon}\cdot \nabla v +\epsilon \int_{D_1}\nabla u_\epsilon\cdot \nabla v=
\int_{D}fv \quad \mbox{ for all } v\in H_0^1(D).
\end{equation}

We try to determine $\{u_i\}_{i=-1}^\infty
\subset H^1_0(D)$ such that, 
\begin{equation}\label{small-eq:expansion}
u_\epsilon=\epsilon^{-1}u_{-1}+u_0+\epsilon u_1+\epsilon^2 u_2+\dots=
\sum_{i=-1}^\infty \epsilon^{i} u_{i}, 
\end{equation}
and such that they satisfy the following Dirichlet boundary 
conditions, 
\begin{equation}\label{small-eq:expansionBC}
u_0=g \mbox{ on } \partial D \quad \mbox{ and }
\quad u_{i}=0 \mbox{ on } \partial D \mbox{ for } i=-1,
\mbox{ and }  i\geq 1.
\end{equation}
Observe that when $u_{-1}\not= 0$, then, $u_{\epsilon}$ 
does not converge when $\epsilon\to 0$. 

If we substitute~\eqref{small-eq:expansion} into~\eqref{small-eq:problem-with-1-inclusion}
we obtain that for all $v\in H^1_0(D)$ we have, 
%\[
%\sum_{i=-1}^\infty \epsilon^{i} \int_{D_0}\nabla u_{i}\cdot \nabla v+
%\sum_{i=-1}^\infty \epsilon^{i+1}\int_{D_1}\nabla u_{i}\cdot \nabla v=\int_{D}fv
%\]
%or 
\[
\epsilon^{-1}\int_{D_0}\nabla u_{0}\cdot \nabla v+
\sum_{i=0}^\infty 
\epsilon^{i}\left(
 \int_{D_0}\nabla u_{i}\cdot \nabla v+\int_{D_1}\nabla u_{i-1}\cdot \nabla v\right)
=\int_{D}fv.
\]
Now we equate powers of $\epsilon$ and analyze all the resulting 
subdommain equations. 

\subsubsection*{Term corresponding to  $\epsilon^{-1}:$} We obtain the equation 

\begin{equation}\label{eq:eps-1}
\int_{D_0}\nabla u_{-1}\cdot \nabla v=0 \mbox{ for all } v\in H^1_0(D).
\end{equation}
Since we assumed $u_{-1}=0$ on $\partial D$, we conclude that 
$\nabla u_{-1}=0$ in $D_0$ and then $u_{-1}^{(0)}=0$  in $D_0$.

\subsubsection*{Term corresponding to $\epsilon^0=1:$} 
We get the equation 
\begin{equation}\label{eq:eps0}
\int_{D_0}\nabla u_{0}\cdot \nabla v+\int_{D_1}\nabla u_{-1}\cdot \nabla v=\int_{D}fv 
\mbox{ for all } v\in H^1_0(D).
\end{equation}
Since $u_{-1}^{(0)}=0$ in $D_0$, we conclude that 
$u_{-1}^{(1)}$ satisfies the following Dirichlet problem in $D_1$, 
\begin{equation}\label{eq:DirProbU-1-eps}
\begin{array}{c}
\int_{D_1} \nabla u_{-1}^{(1)}\cdot \nabla z =\int_{D_1}fz \quad \mbox{ for all } z\in 
H^1_0(D_1)\\\\
u_{-1}^{(1)}=0 \mbox{ on } \partial D_1.
\end{array}
\end{equation}

Now we compute $u_{0}^{(0)}$ in $D_0$. As before, from~\eqref{eq:eps0}, 
\[
\nabla u_0^{(0)}\cdot n_0=-\nabla u_{-1}^{(1)}\cdot n_1 \mbox{ on } 
\partial D_1.
\]
Then we can obtain $u_0^{(0)}$ in $D_0$ by solving 
the following problem
\begin{equation}\label{eq:Neumanu1-eps}
\begin{array}{c}
\int_{D_0}\nabla u_0^{(0)}\cdot \nabla z=\int_{D_0}fz -
\int_{\partial D_1 }\nabla u_{-1}^{(1)}\cdot n_1  z
\quad \forall z\in H^1(D_0) \mbox{ with }
z=0 \mbox{ on } \partial D,\\\\
u_{0}^{(0)}=g \mbox{ on } \partial D\subset \partial D_0.
\end{array}
\end{equation}
%Here we use $H^1_0(D_1, \partial D)$ to denote the 
%set of function in $H^1_0(D_0)$ with zero trace on 
%$\partial D$. 

\subsubsection*{Term corresponding to $\epsilon^i$ with $i\geq 1$:}
We get the equation
\[
 \int_{D_0}\nabla u_{i}\cdot \nabla v+\int_{D_1}\nabla u_{i-1}\cdot \nabla v=0
\]
which implies 
that $u_i^{(1)}$ is harmonic in $D_1$ for all $i\geq 0$ and that 
$u_i^{(0)}$ is harmonic in $D_0$ for $i\geq 1$. Also, 
\[
\nabla u_{i}^{(0)}\cdot n_0=-\nabla u_{i-1}^{(1)}\cdot n_1.
\]

Given $u_{i-1}^{(0)}$ in $D_0$ (e.g., $u_{0}$ in $D_0$ above) 
we can find 
$u_{i-1}^{(1)}$ in $D_1$ by  solving
 the Dirichlet problem with the known Dirichlet data, 
\begin{equation}\label{eq:DirforuiD0-eps}
\begin{array}{c}
 \int_{D_1}\nabla u_{i-1}^{(1)}\cdot \nabla z=0 \mbox{ for all } 
z\in H^1_0(D_1)  \\\\
u_{i-1}^{(1)}=u_{i-1}^{(0)}  \mbox{ on } \partial D_1.
\end{array}
\end{equation}

To find 
$u_{i}^{(0)}$ in $D_0$ we solve the problem
\begin{equation}\label{eq:Neuforuiplu1D1-eps}
\begin{array}{c}\displaystyle
 \int_{D_0}\nabla u_{i}^{(0)}\cdot \nabla z=-\int_{\partial D_0 }\nabla 
u_{i-1}^{(1)}\cdot n_1z \quad \mbox{ for all } 
z\in H^1(D_0) \mbox{ with } z=0 \mbox{ on } \partial D,\\\\
u_{i}^{(0)}=0\mbox{ on } \partial D.
\end{array}
\end{equation}

\subsection{Convergence in $H^1(D)$}
In this section we study the convergence of 
the expansion~\eqref{small-eq:expansion}-\eqref{small-eq:expansionBC}.
The following lemma is obtained using classical estimates 
and trace theorems in the involved subdomains. 
\begin{lemma}
Let $u_{-1}$ vanish in $D_0$ and be defined using 
problem~\eqref{eq:DirProbU-1-eps} in $D_1$. We have 
\[
\|u_{-1}\|_{H^1(D_1)}\preceq \|f\|_{H^{-1}(D_1)}
\]
and 
\[
\|u_{0}\|_{H^1(D_0)}\preceq  \|f\|_{H^{-1}(D_0)}+
\|u_{-1}\|_{H^1(D_1)}+\|g\|_{H^{1/2}(\partial D)}.
\]
Moreover, if we consider problems~\eqref{eq:DirforuiD0-eps}
 and~\eqref{eq:Neuforuiplu1D1-eps}, we have that 
for $i\geq 1$, 
$
\|u_{i}\|_{H^1(D_1)}\preceq \|u_{i}\|_{H^1(D_0)},
$
$
\|u_{i}\|_{H^1(D_0)}\preceq \|u_{i-1}\|_{H^1(D_1)}
$
and 
\[
\|u_{i}\|_{H^1(D)}\preceq \|u_{i-1}\|_{H^1(D_1)}.
\]
\end{lemma}

The convergence of the expansion follows.
\begin{theorem}
There is a constant $C>0$ such that for every 
$\epsilon<1/C$, the expansion~\eqref{small-eq:expansion} converges 
(absolutely) in $H^1(D)$.
\end{theorem}
\begin{proof}
 There is a constant $C$ such that, for every $i\geq 1$ we have 
\begin{eqnarray}
\|u_{i}\|_{H^1(D)}&\leq& C\|u_{i-1}\|_{H^1(D_1)}\leq
C\|u_{i-1}\|_{H^1(D)} \\
&\leq & \dots\\
&\leq& C^{i}\|{u}_0\|_{H^1(D_1)}
\end{eqnarray}
and then 
\[
\|\sum_{i=1}^\infty \epsilon^{i}u_{i}\|_{H^1(D)}\leq 
\frac{\|{u}_1\|_{H^1(D_0)} }{C}\sum_{i=1}^\infty\left(C\epsilon\right)^{i}.
\]
The last series converges when $\epsilon<1/C$. Using the bound 
for $u_0,u_1$ and $u_{-1}$ we obtain that there is a constant $C_1$ such that 

\[
\|\sum_{i=0}^\infty \epsilon^{i}u_{i}\|_{H^1(D)}\leq 
C_1 (\|f\|_{H^{-1}(D_0)}+
\|f\|_{H^{-1}(D_1)}+\|g\|_{H^{1/2}(\partial D)})\sum_{i=1}^\infty\left(C\epsilon\right)^{i}.
\]

\end{proof}
\begin{corollary}\label{cor:convergence}
There are positive constants $C_1$ and $C$ such that for every 
$\epsilon<1/C$, we have 
\begin{eqnarray*}
&&\|u-u_{-1}-\sum_{i=1}^I \epsilon^{i}u_{i}\|_{H^1(D)} \leq \\
&&C_1(\|f\|_{H^{-1}(D_0)}+
\|f\|_{H^{-1}(D_1)}+\|g\|_{H^{1/2}(\partial D)})\sum_{i=I+1}^\infty\left(C\epsilon\right)^{i},
\end{eqnarray*}
for $I\geq 0$.
\end{corollary}

Using the asymptotic 
expansion~\eqref{small-eq:expansion}-\eqref{small-eq:expansionBC} 
we can write an asymptotic problem for $\epsilon\to 0$.
\begin{corollary}\label{cor:cor}
If $f=0$ in $D_1$, we can write
\[
u_\epsilon=u_0+\epsilon u_1+\epsilon^2 u_2+\dots
\]
where $u_0^{(0)}=u_0|_{D_0}$ satisfy 
the following problem 
with Dirichlet data on $\partial D$ 
 and zero Neumann data on $\partial D_1$.
\[
\begin{array}{c}
\int_{D_0}\nabla u_0^{(0)}\cdot \nabla z=\int_{D_0}fz 
\quad \forall z\in H^1_0(D_0) \mbox{ with }
z=0 \mbox{ on } \partial D,\\\\
u_{0}^{(0)}=g \mbox{ on } \partial D\subset \partial D_0.
\end{array}\]
Additionally, we can find 
$u_{0}^{(1)}=u_0|_{D_1}$ by extending harmonically to $D_1$ the 
known Dirichlet data on $\partial D$, that is, 
\begin{equation}
 \int_{D_1}\nabla u_{i-1}^{(1)}\cdot \nabla z=0 \mbox{ for all } 
z\in H^1_0(D_1)  
\end{equation}
with 
\[
u_{0}^{(1)}=u_{0}^{(0)}  \mbox{ on } \partial D_1.
\]
The series converges absolutely in $H^1(D)$ for $\epsilon$ 
sufficiently small.

\end{corollary}

\begin{figure}[htb]
\centering
{\psfig{figure=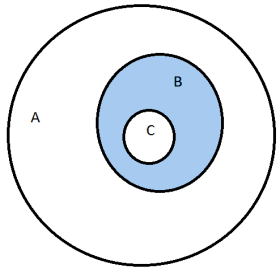,height=4cm,width=4cm,angle=0}}
\caption{Example of disconnected background.
Here $D_0=A\cup C$ and $D_1=B$ is a low-conductivity 
inclusion. See Remark~\ref{rem:disconnected-inclusions}. }
\label{fig:nonconnected-low} 
\end{figure}

When $D_0$ is not connected, the following observation can be made.

\begin{remark}
\label{rem:disconnected-inclusions}
In the case of $D_0$ being disconnected, we have 
that~\eqref{eq:eps-1} implies that $u_{-1}$ 
is constant in each connected component of $D_0$ and it vanishes 
only in the connected components whose boundary intersects 
$\partial D$. The function $u_{-1}$ will be constant 
in the other interior connected components. 
This is similar to the case of high-conductivity 
inclusions. The function $u_{-1}$ will be zero 
only if the forcing term vanishes in these interior 
components also; see Equations~\eqref{ref:constantcu0}
and~\eqref{eq:def:Yi}. For instance, 
consider the case illustrated in 
Figure~\ref{fig:nonconnected-low} where 
the background domain are $D_0=A\cup C$
and the inclusions is given by $D_1=B$. 
In this case it is easy to see that 
$u_{-1}$ satisfies a problem similar 
to problem~\eqref{eq:limitproblem} in $B\cup \overline{C}$.
Then, $u_{-1}$ will vanish only if the forcing 
term vanishes in $B\cup \overline{C}$.
In this case, a result similar to Corollary~\ref{cor:cor}
above can be stated.
\end{remark}

\subsection{Multiple low-conductivity inclusions}\label{small-sec:multiple}

Let $\kappa$ be defined by 
\begin{equation}\label{eq:coeff1inc-eps-multiple}
\kappa(x)=\left\{\begin{array}{cc} 
\epsilon& x\in   D_m, ~~m=1,\dots,M, \\
1,& x\in  D_0=D\setminus \cup_{m=1}^M\overline{D}_m,
\end{array}\right.
\end{equation}
and denote by $u_{\epsilon}$
 the solution of~\eqref{eq:problem} with coefficient 
(\ref{eq:coeff1inc-eps-multiple}).
 We assume that 
$D_i$ is compactly included in the open set $D\setminus 
\cup_{\ell=1, \ell \not=m}^M\overline{D}_\ell$, i.e., 
$\overline{D}_m\subset D\setminus \cup_{\ell=1, \ell\not=m}^M \overline{D}_\ell$, and we define $D_0:= D\setminus 
\cup_{m=1}^M\overline{D}_m$.

Expansion~\eqref{small-eq:expansion}-\eqref{small-eq:expansionBC}
extends easily to this case of multiple low-conductivity inclusions, that is,
\begin{itemize}
\item We have $u_{-1}^{(0)}=0$ on $D_0$. Also, that each,
$u_{-1}^{(\ell)}$ satisfies the following Dirichlet problem in $D_\ell$, 
\begin{equation}\label{eq:DirProbU-1-eps-multiple}
\begin{array}{c}
\int_{D_\ell} \nabla u_{-1}^{(\ell)}\cdot \nabla z =\int_{D_\ell}fz \quad \mbox{ for all } z\in 
H^1_0(D_\ell)\\\\
u_{-1}^{(\ell)}=0 \quad \mbox{ on } \partial D_\ell.
\end{array}
\end{equation}

\item We can obtain $u_0^{(0)}$ in $D_0$ by solving 
the problem
\begin{equation}\label{eq:Neumanu1-eps-multiple}
\begin{array}{c}
\displaystyle \int_{D_0}\nabla u_0^{(0)}\cdot \nabla z=\int_{D_0}fz -
\sum_{\ell=1}^{M}\int_{\partial D_\ell }\nabla u_{-1}^{(\ell)}\cdot n_1  z
 \quad  \forall z\in H^1(D_0), 
z|_{\partial D}=0,\\\\
u_{0}^{(0)}=g \mbox{ on } \partial D.
\end{array}
\end{equation}

\item Finally, we have that 
 $u_i^{(\ell)}$ is harmonic in $D_\ell$ for all $i\geq 0$ and that 
$u_i^{(0)}$ is harmonic in $D_0$ for $i\geq 1$.

Given $u_{i-1}^{(0)}$ in $D_0$,  
we can find 
$u_{i-1}^{(\ell)}$ in $D_\ell$ by  solving
 the following Dirichlet problem with the known Dirichlet data, 
\begin{equation}\label{eq:DirforuiD0-eps-multiple}
\begin{array}{c}\displaystyle 
 \int_{D_\ell}\nabla u_{i-1}^{(\ell)}\cdot \nabla z=0 \mbox{ for all } 
z\in H^1_0(D_\ell)  \\
u_{i-1}^{(\ell)}=u_{i-1}^{(0)}  \mbox{ on } \partial D_\ell.
\end{array}
\end{equation}

To find 
$u_{i}^{(0)}$ in $D_0$ we solve the problem
\begin{equation}\label{eq:Neuforuiplu1D1-eps-multiple}
\begin{array}{c}
\displaystyle \int_{D_0}\nabla u_{i}^{(0)}\cdot \nabla z=-
\sum_{\ell=1}^M\int_{\partial D_\ell }\nabla 
u_{i-1}^{(\ell)}\cdot n_\ell z \quad \forall  
z\in H^1_0(D_0) \mbox{ with } z|_{\partial D}=0,\\\\
u_{i}=0 \mbox{ on } \partial D\subset \partial D_0.
\end{array}
\end{equation}

\end{itemize}
The convergence of the expansion~\eqref{small-eq:expansion}-\eqref{small-eq:expansionBC}
is similar to the case of one low-conductivity inclusion. In 
particular Corollary~\ref{cor:convergence} holds in this case.
%\begin{corollary}
%There are  positive constants $C_1$ and $C$ such that for every 
%$\epsilon<1/C$, we have 
%\[
%\|u-u_{-1}-\sum_{i=0}^I \epsilon^{i}u_{i}\|_{H^1(D)}\leq 
%\frac{\|{u}_1\|_{H^1(D_0)} }{C}\sum_{i=I+1}^\infty\left(C\epsilon\right)^{i},
%\]
%for $I\geq 2$.
%\end{corollary}

\section{An example with low- and high- conductivity inclusions}\label{sec:low-and-high}

In this section we show  an 
example  with a high- and a low-conductivity inclusion.   
The procedures were introduced in detail in
Sections~\ref{sec:high} and~\ref{sec:low}. 
We show only how  to write the subdomains problems for the 
leading terms of the expansion. 

Consider $\kappa$ to be defined by 
\begin{equation}\label{hl-eq:coeff1inc}
\kappa(x)=\left\{\begin{array}{cc} 
\eta,& x\in   D_1,  \\
{1}/{\eta},& x \in   D_2,  \\
1,& x\in  D_0=D\setminus (\overline{D}_1\cup \overline{D}_2).
\end{array}\right.
\end{equation}

As before we write $
u_\eta=\eta u_{-1}+u_0+\frac{1}{\eta}u_1+\frac{1}{\eta^2}u_2+\dots=
\sum_{i=-1}^\infty \eta^{-i} u_{i}, $
with $u_0=g$ on  $\partial D$ and 
$ u_{i}=0$ on $\partial D$ for  $i\not = 0$.
We need that \begin{eqnarray*}
&&\eta^2\int_{D_1} \nabla u_{-1}\cdot \nabla v+
\eta\left( \int_{D_0}\nabla u_{-1}\cdot \nabla v+
\int_{D_1} \nabla u_{0} \cdot \nabla v\right)+\\
&&\quad \sum_{i=0}^\infty 
\eta^{i}\left(
 \int_{D_0}\nabla u_{i}\cdot \nabla v+\int_{D_1}\nabla u_{i+1}\cdot \nabla v+
\int_{D_2}\nabla u_{i-1}\cdot \nabla v\right)
=\int_{D}fv.
\end{eqnarray*}
Now we equate powers and analyze the subdomain equations.
We assume  that 
$D_i$ is connected and compactly included in $D$ ($\overline{D}_i\subset D$), 
$i=1,2$.  
We also assume that the distance between $D_1$ and  $D_2$
is strictly positive, then 
 $u_{-1}^{(0)}=0$, $u_{-1}^{(1)}=0$ and  
$u_{-1}^{(2)}$ solves the following Dirichlet problem in $D_1$, 
\begin{equation}\label{lh-eq:DirProbU-1-eps}
\begin{array}{c}
\int_{D_1} \nabla u_{-1}^{(2)}\cdot \nabla z =\int_{D_1}fz \quad \mbox{ for all } z\in 
H^1_0(D_1)\\\\
u_{-1}^{(2)}=0 \mbox{ on } \partial D_1.
\end{array}
\end{equation}
This defines the function $u_{-1}$. 
To write a problem for $u_{0}$, let 
\[
V_{const}=\{ v\in H^1(D\setminus \overline{D}_2), \mbox{ such that } 
v=0 \mbox{ on } \partial D \mbox{ and }  v^{(1)}=v|_{D_1} \mbox{ is  constant }\}.
\]
We have that 
\begin{equation}\label{hl-eq:limitproblem}
\begin{array}{c} 
\int_{D_0}\nabla u_{0}\cdot \nabla z=\int_{D}fz -
\int_{\partial D_2} \nabla u_{-1}^{(2)}\cdot n_2 z 
\mbox{ for all } z\in V_{const},\\\\
u_0=g \mbox{ on } \partial D.
\end{array}
\end{equation}

The problem above can be analyzed using the
 \emph{harmonic characteristic function} 
$\chi_{D_1}\in H^1_0(D)$ defined in~\eqref{eq:def:chiDl}
with $M=2$. The solution of the asymptotic 
problem above gives $u_0^{(0)}$ and the constant function
$u_0^{(1)}$. As before,  the constant 
$u_0^{(1)}$ can be determined explicitly using 
an expression similar to~\eqref{ref:constantcu0-v2}. 
 To complete the definition 
of $u_{0}$ we observe that $u_{0}^{(2)}$ satisfies 
the following Dirichlet problem with the known Dirichlet data, 
\begin{equation}\label{hl-eq:DirforuiD0-eps}
\begin{array}{c}
 \int_{D_2}\nabla u_{0}^{(2)}\cdot \nabla z=0 \mbox{ for all } 
z\in H^1_0(D_1)  \\\\
u_{0}^{(2)}=u_{0}^{(0)}  \mbox{ on } \partial D_1.
\end{array}
\end{equation}

The functions $u_{i}$, $i=1,\dots$, can be determined form the 
equation 
\[
 \int_{D_0}\nabla u_{i}\cdot \nabla v+\int_{D_1}\nabla u_{i+1}\cdot \nabla v+
\int_{D_2}\nabla u_{i-1}\cdot \nabla v=0 \quad \mbox{ for all } v\in H^1_0(D).
\]
This procedure is similar to the ones developed before
and   presented in detail 
in Sections~\ref{sec:high} and~\ref{sec:low}. As 
before, $u_{i}$ is  harmonic in each region. 
Its restriction to subregions can be determined by solving 
subdomain problems involving Dirichlet, Neumann  or mixed boundary 
conditions on the inclusions boundaries.

\section{Conclusions and comments}\label{sec:conclusions}

We use asymptotic expansions to study high-contrast 
problems. 
We  derive and analyze asymptotic power series for high-contrast elliptic problems. We mostly consider the case 
of binary media with interior isolated inclusions. 
High- or low-conductivity inclusion configurations are considered.
The  coefficients in the expansions are determined 
sequentially  by a Dirichlet-to-Neumann procedure. In the case 
of high-conductivity inclusions, 
the Neumann problem needs to satisfy a 
compatibility of fluxes. This flux-compatibility condition is obtained using
an auxiliary finite dimensional projection
problem. The related finite dimensional space 
is spanned by harmonic extension of characteristic functions
of each subdomain.

{The 
asymptotic  limits when the 
contrast increases to infinity are  recovered and analyzed}; 
see Theorems~\ref{thm:convergence-one-high} and 
\ref{thm:convergence-multiple-high} and Corollary 
\ref{cor:cor}. 
The convergence of the expansions in $H^1(D)$ 
is obtained provided 
that the contrast is larger than a constant $C$ that 
depends on the background domains and the domains 
representing the inclusions. The convergence rate is algebraic.
We consider the 
case of isolated interior inclusions which can be high and low
conductivities.
Other more complex configurations  can be analyzed
following a similar procedure. The analysis 
covers the cases with high- and low-conductivity inclusions. 
See   Figure~\ref{fig:figure2} for schematic representations
of two dimensional 
configurations. If the low-conductivity value 
is $\epsilon$ and the high-conductivity 
value is $\eta=1/\epsilon$, then, an 
expansion similar to~\eqref{small-eq:expansion}-\eqref{small-eq:expansionBC} 
can be used for all the examples in Figure~\ref{fig:figure2}.
For general values of $\eta(\epsilon)=\epsilon^{-\rho}$, 
an 
expansion similar to~\eqref{small-eq:expansion}-\eqref{small-eq:expansionBC} 
can also be derived where the coefficients in front of spatial terms
will scale as $\epsilon^{l + m\rho}$, where $l\geq -1$ and $m\geq 0$ are integers.
\begin{figure}[htb]
\centering
{\psfig{figure=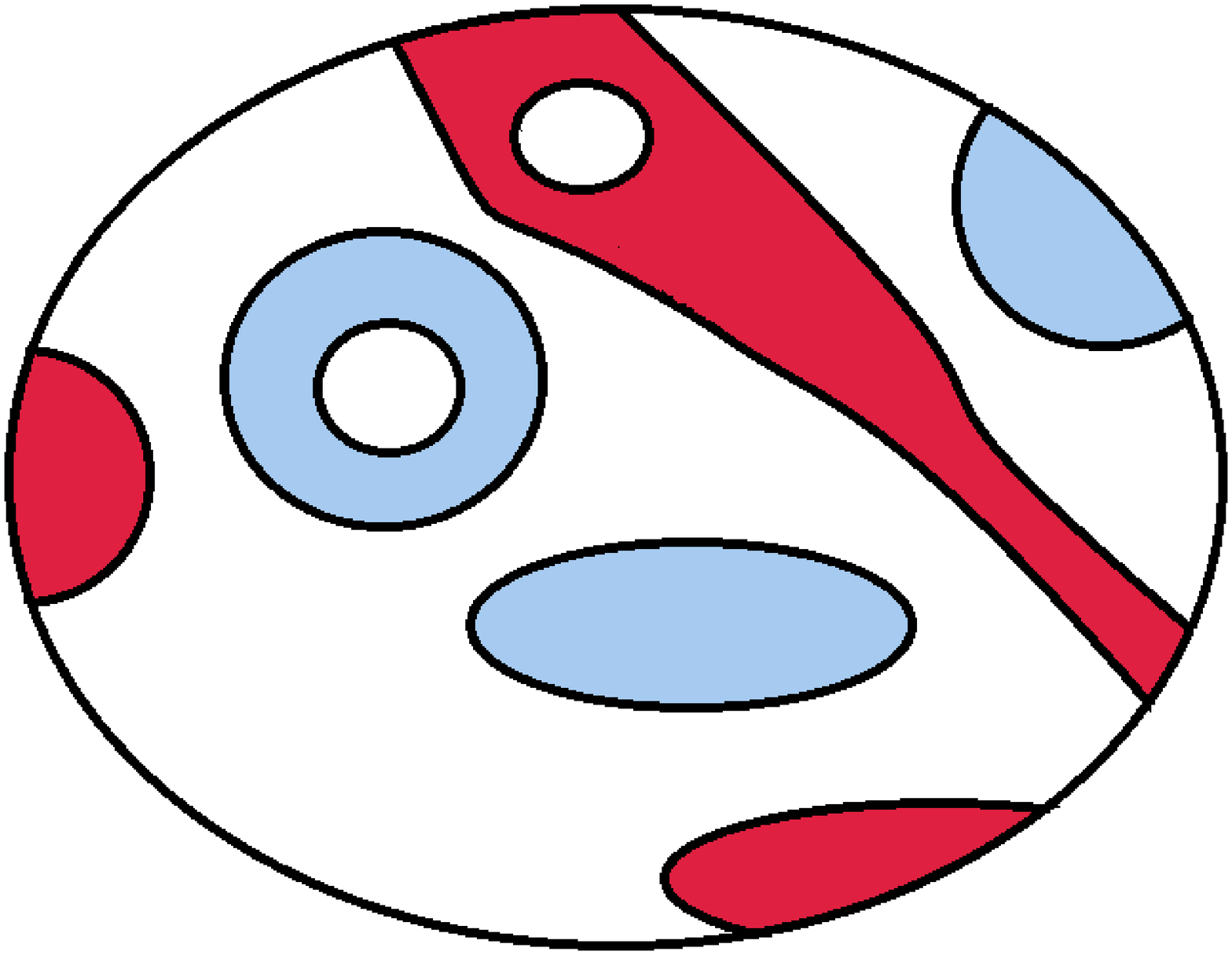,height=3cm,width=4cm,angle=0}}
{\psfig{figure=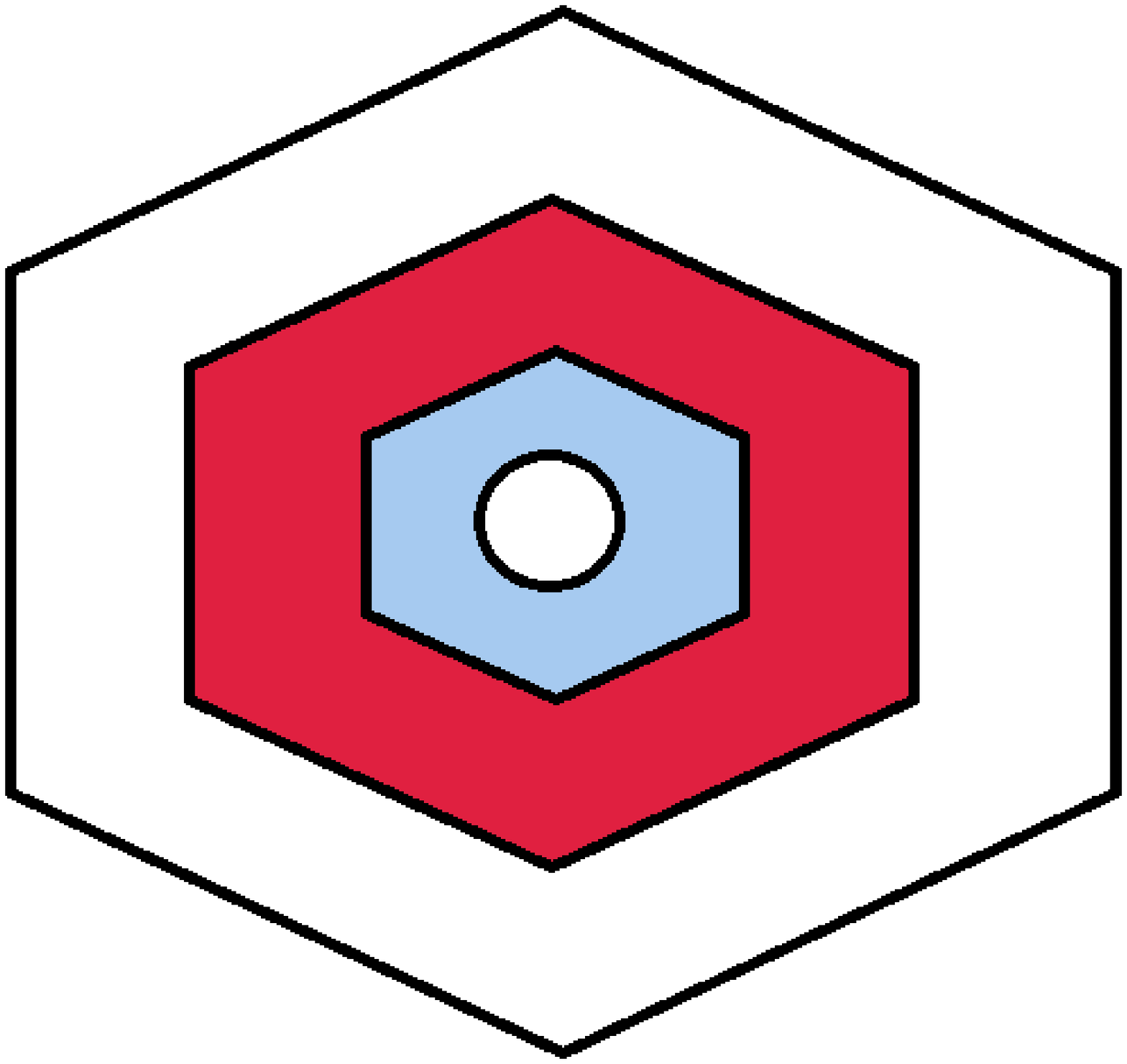,height=3cm,width=4cm,angle=0}}
{\psfig{figure=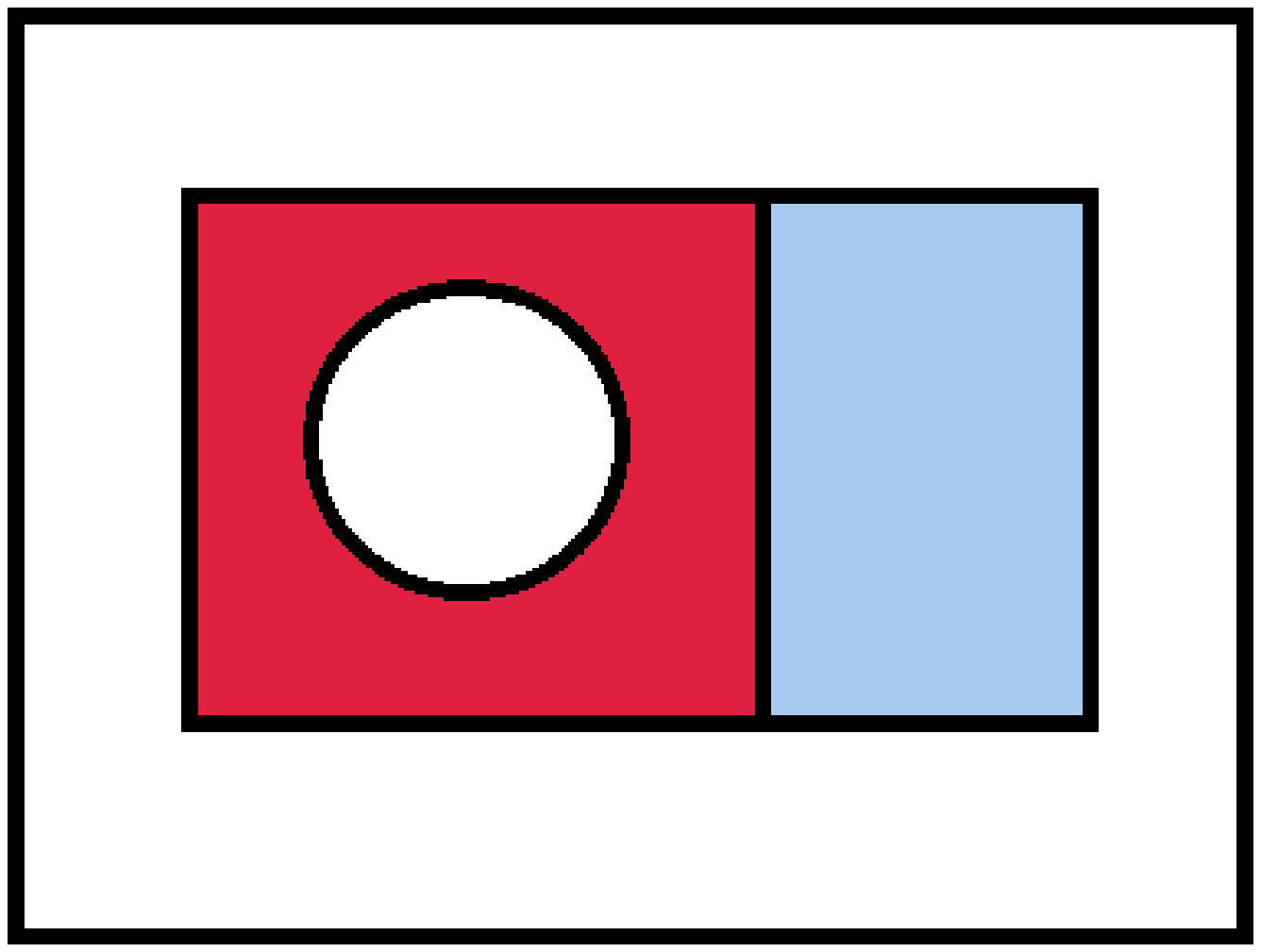,height=3cm,width=4cm,angle=0}}
{\psfig{figure=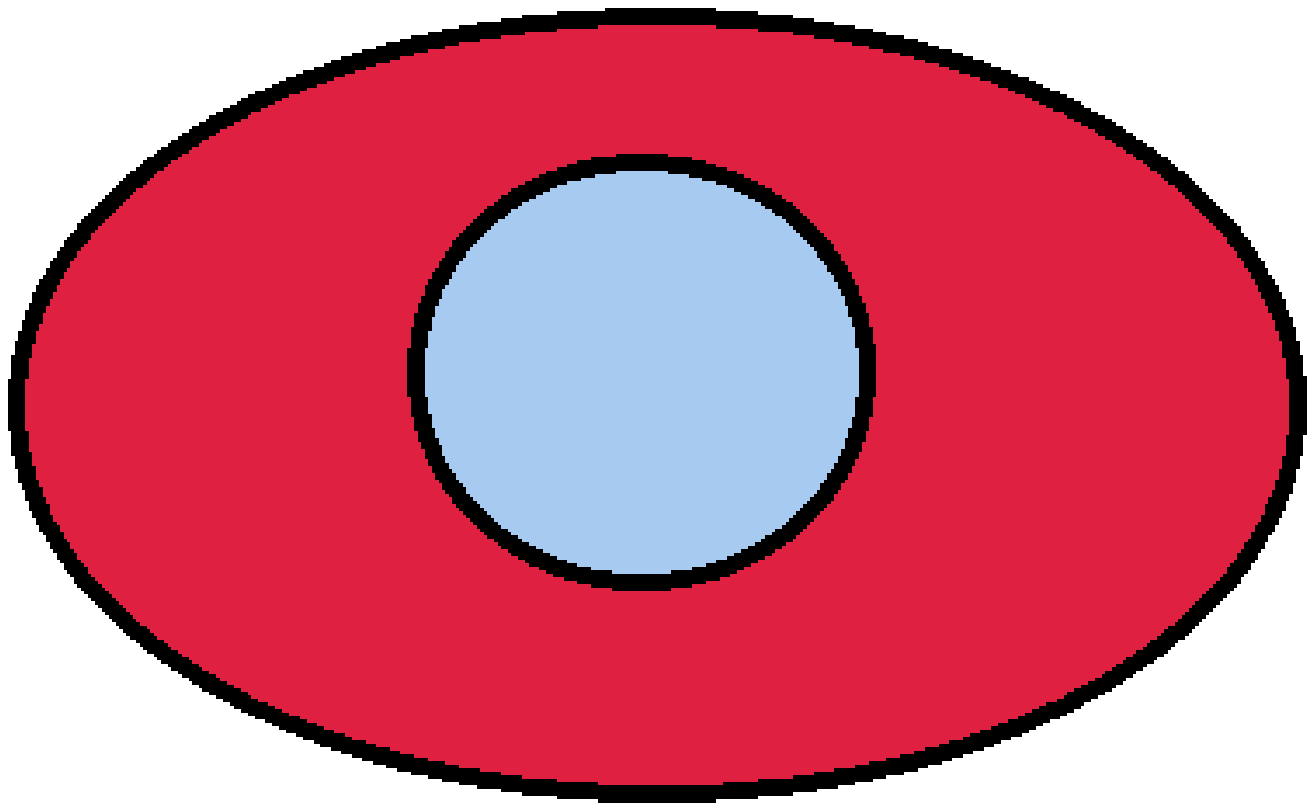,height=3cm,width=4cm,angle=0}}
\caption{Four different configurations with  
high-conductivity $1/\epsilon$ (red) and low-conductivity 
$\epsilon$ (blue) regions. 
Expansion~\eqref{small-eq:expansion} 
can be obtained in these cases. }
\label{fig:figure2} 
\end{figure}

More general coefficients can also be studied. 
Similar expansions for other problems  
related with flows in high-contrast multiscale 
media can be obtained, e.g., models like 
heat conduction, wave propagation, Darcy or Brinkman flow, and elasticity 
problems. 
Efficient solution techniques for solving the system of linear
equations~\eqref{eq:Ageom} will be a subject of future research, 
in particular, localization procedures 
for the harmonic characteristic functions will be studied.
Questions concerning the convergence of the 
series in stronger norms as well as 
computing quantities of interest will be studied 
in the future. Reduced contrast approximation and 
related multiscale methods as in~\cite{ce10,cgh09} will be the subject 
of future studies.

\section*{Acknowledgments}
This publication is based in part on work supported by Award No. KUS-C1-016-04, made by King Abdullah University of Science and Technology (KAUST).
\bibliographystyle{plain}

\begin{thebibliography}{1}




\bibitem{aarnes}
J. Aarnes and T. Hou,
{\em Multiscale domain decomposition methods for elliptic problems with
  high aspect ratios},
 { Acta Math. Appl. Sin. Engl. Ser.}, 18(1):63-76, 2002.



\bibitem{bl10}
{  I. Babu{\v{s}}ka  and R. Lipton}, 
{\em Optimal Local approximation spaces for generalized finite
element methods with application to multiscale problems}, submitted.




\bibitem{blp78}
A. Bensoussan, J. L. Lions, and G. Papanicolaou,
{\em Asymptotic analysis for periodic structures,}
Volume 5 of Studies in Mathematics and Its Applications,
North-Holland Publ., 1978.


\bibitem{BerlyandNovikov}
{L. Berlyand and A. Novikov}, {\em Error of the network approximation for densely packed composites with irregular geometry},  SIAM Journal on Mathematical Analysis, 34(2) (2002).   385-408.

\bibitem{yulya-berlyand}
Berlyand, Leonid; Cardone, Giuseppe; Gorb, Yuliya; Panasenko, Gregory. 
{\it Asymptotic analysis of an array of closely spaced absolutely conductive inclusions}. Netw. Heterog. Media 1 (2006), no. 3, 353--377. MR2247782 (2007m:35008) Add to clipboard 
\bibitem{yulya-berlyan2}
Berlyand, Leonid; Gorb, Yuliya; Novikov, Alexei. 
{\it Discrete network approximation for highly-packed composites with irregular geometry in three dimensions}. Multiscale methods in science and engineering, 21--57, Lect. Notes Comput. Sci. Eng., 44, Springer, Berlin, 2005.

\bibitem{bo09}
{L. Berlyand and H. Owhadi}, {\em A new approach to homogenization
with arbitrary rough high contrast coefficients for
scalar and vectorial problems}, submitted.

\bibitem{bo09_1}
{L. Berlyand and H. Owhadi}, {\em Flux norm approach  to finite dimensional homogenization approximations with  non-separated scales and high contrast},
Archives for Rational Mechanics and Analysis (2010, Volume 198, Number 2, 677-721).


\bibitem{bp98}
L. Borcea and G.C. Papanicolaou,  {\em Network approximation for transport properties of high contrast materials}, SIAM Journal on Applied Mathematics, vol. 58, no. 2, 1998,   501-539.

\bibitem{ceg10}
{V. M. Calo, Y. Efendiev, and J. Galvis}, {\em A note on vatiational
multiscale methods for high-contrast heterogeneous flows with 
rough source terms}, 34 (9), September 2011, pp. 1177-1185. 



\bibitem{AMG1}
T. Chartier, R. Falgout, V.E. Henson, J. Jones, T. Manteuffel, S. McCormick, J. Ruge, and P.S. Vassilevski, 
{\em Spectral element agglomerate AMGe}, in Domain Decomposition Methods in Science and Engineering XVI, Lecture Notes in Computational Science and Engineering, Springer-Verlag, Berlin Heidelberg {\bf 55}(2007), 515-524.




\bibitem{cgh09}
{C.C. Chu, I.G. Graham, and T.Y. Hou}, {\em A new multiscale
 finite element 
method for high-contrast elliptic interface problem}, Math. Comp., 79 , 1915-1955, 2010. 

\bibitem{ce10}
E. Chung and Y. Efendiev, {\em Reduced-contrast approximations for high-contrast multiscale flow problems},  Multiscale Model. Simul.  8  (2010),  no. 4, pp. 1128-1153.

\bibitem{DangQuang}
Dang Quang A. {\it Approximate method for solving an elliptic problem with discontinuous coefficients}. J. Comput. Appl. Math. 51 (1994), no. 2, 193--203. 

\bibitem{Drjya}
M. Dryja, {\em Multilevel Methods for Elliptic Problems with Discontinuous
                 Coefficients in Three Dimensions}, Seventh International Conference of Domain Decomposition
                 Methods in Scientific and Engineering Computing, by David E. Keyes and Jinchao Xu, vol. 180, 1994, 43-47.


\bibitem{ge09_2}
 Y. Efendiev and J. Galvis, {\em A domain decomposition preconditioner for multiscale high-contrast problems}, in Domain Decomposition Methods in Science and Engineering XIX, Huang, Y.; Kornhuber, R.; Widlund, O.; Xu, J. (Eds.), Volume 78 of Lecture Notes in Computational Science and Engineering, Springer-Verlag, 2011, Part 2,   189-196.

\bibitem{egvdd20}
Y. Efendiev, J. Galvis and P. Vassielvski, {\em Spectral element agglomerate algebraic multigrid methods for elliptic problems with 
high-Contrast coefficients}, in Domain Decomposition Methods in Science and Engineering XIX, Huang, Y.; Kornhuber, R.; Widlund, O.; Xu, J. (Eds.), Volume 78 of Lecture Notes in Computational Science and Engineering, Springer-Verlag, 2011, Part 3,   407-414. 

\bibitem{eglw11}
Y. Efendiev, J. Galvis, R. Lazarov and J. Willems, {\em Robust domain decomposition preconditioners for abstract symmetric positive definite bilinear forms}, submitted.



\bibitem{egw10}
{  Y. Efendiev, J. Galvis, and X. H. Wu}, {\em Multiscale finite element methods for high-contrast problems using local spectral basis functions}, Journal of Computational Physics. Volume 230, Issue 4, 20 February 2011, Pages 937-955. 


\bibitem{eh09}
{Y. Efendiev and T. Hou}, {\em Multiscale finite element methods. Theory
and applications}, Springer, 2009.



\bibitem{EILRW08}
R.E. Ewing, O.~Iliev, R.D. Lazarov, I.~Rybak, and J.~Willems.
\newblock A simplified method for upscaling composite materials with high
  contrast of the conductivity.
\newblock {\em SIAM Journal on Scientific Computing}, 31(4):2568--2586, 2009.


\bibitem{ge09_1}
 J. Galvis and Y. Efendiev, {\it Domain decomposition preconditioners for
   multiscale flows in high contrast media},  SIAM MMS, Volume 8, Issue 4,
   1461-1483 (2010).


\bibitem{ge09_3}
 J. Galvis and Y. Efendiev, {\it
 Domain decomposition  preconditioners for multiscale flows in high-contrast media: Reduced dimension coarse spaces}, SIAM MMS, Volume 8, Issue 5,   1621-1644 (2010).



\bibitem{Graham1}
I. G. Graham, P. O. Lechner, and R. Scheichl,
{\em Domain decomposition for multiscale {PDE}s},
 { Numer. Math.}, 106(4):589-626, 2007.

\bibitem{Grisvard}
Grisvard, P. {\it Elliptic problems in nonsmooth domains}.
 Monographs and Studies in Mathematics, 24. Pitman (Advanced Publishing Program), Boston, MA, 1985. 

\bibitem{hw97}
{  T.Y. Hou and X.H. Wu}, {\em A multiscale finite element method for
  elliptic problems in composite materials and porous media}, Journal of
  Computational Physics, 134 (1997),   169-189.
\bibitem{klapper}
{ Klapper, I.; Shaw, T.} {\it A large jump asymptotic framework for solving elliptic and parabolic equations with interfaces and strong coefficient discontinuities}. Appl. Numer. Math. 57 (2007), no. 5-7, 657--671

\bibitem{Tarekbook}
T. P. A. Mathew,
 {\em Domain decomposition methods for the numerical solution of
  partial differential equations}, volume 61 of {\em Lecture Notes in
  Computational Science and Engineering}, Springer-Verlag, Berlin, 2008.


\bibitem{nataf11}
F. Nataf, H. Xiang, V. Dolean, N. Spillane, {\em A coarse space construction based on local Dirichlet to Neumann maps}, accepted for publication in SIAM J. Sci Comput., 2011.

\bibitem{Nepom-91}
S.V. Nepomnyaschikh, {\em Mesh theorems on traces, normalizations of
function traces and their inversion},
Soviet J. Numer. Anal. Math. Modelling, 6(2):151-168, 1991.


\bibitem{oz11}
{  H. Owhadi and L. Zhang},
{\em Localized bases for finite dimensional homogenization approximations with  non-separated scales and high-contrast}, submitted to SIAM MMS, Available at Caltech ACM Tech Report No 2010-04.  arXiv:1011.0986 

\bibitem{pavliotis-stuart}
{ Pavliotis, Grigorios A.; Stuart, Andrew M.}
{\em  Multiscale methods. Averaging and homogenization}.
 Texts in Applied Mathematics, 53. Springer, New York, 2008. 

\bibitem{rob_clement}
{ C. Pechstein and R. Scheichl}, 
{\em Analysis of FETI methods for multiscale PDEs}, Numerische Mathematik 111(2):293-333, 2008.


\bibitem{marcus1}
M. Sarkis,  
{\em Nonstandard coarse spaces and Schwarz methods for elliptic problems with discontinuous coefficients using non-conforming elements},
  Numer. Math., 77(3), 383-406, 1997.


\bibitem{tw}
A. Toselli and O. Widlund. {\em Domain decomposition methods---algorithms and theory}, volume 34
  of { Springer Series in Computational Mathematics}, Springer-Verlag, Berlin, 2005.

\bibitem{vsg}
J. Van lent,  R. Scheichl, and I.G. Graham, 
{\em Energy-minimizing coarse spaces for two-level Schwarz methods for multiscale PDEs}. Numer. Linear Algebra Appl. 16 (2009), no. 10, 775-799.

\bibitem{AMG2}
P.S. Vassilevski,
{\em "Coarse Spaces by Algebraic Multigrid: Multigrid Convergence and Upscaling Error Estimates},
(2010) (to appear).
Available as Lawrence Livermore National Laboratory Technical Report LLNL-PROC-432896, May 21, 2010.



\bibitem{XuZikatanov}
J. Xu and L. Zikatanov, {\em On an energy minimizing basis for algebraic multigrid
methods}, Comput. Visual Sci., 7:121-127, 2004.



\end{thebibliography}

\def\cprime{$'$}

\end{document}